\documentclass[11pt,reqno,letterpaper]{amsart}
\usepackage{color}
\usepackage[colorlinks=true, allcolors=blue,backref=page]{hyperref}
\usepackage{amsmath, amssymb, amsthm}
\usepackage{mathrsfs}
\usepackage{mathtools}
\usepackage[noabbrev,capitalize,nameinlink]{cleveref}
\crefname{equation}{}{}
\usepackage{fullpage}
\usepackage[noadjust]{cite}
\usepackage{graphics}
\usepackage{pifont}
\usepackage{tikz}
\usepackage{bbm}
\usepackage[T1]{fontenc}

\usetikzlibrary{arrows.meta}

\usepackage{environ}
\usepackage{framed}
\usepackage{url}
\usepackage[linesnumbered,ruled,vlined]{algorithm2e}
\usepackage[noend]{algpseudocode}
\usepackage[labelfont=bf]{caption}
\usepackage{cite}
\usepackage{framed}
\usepackage[framemethod=tikz]{mdframed}
\usepackage{appendix}
\usepackage{graphicx}
\usepackage[textsize=tiny]{todonotes}
\usepackage{tcolorbox}
\usepackage{enumerate}
\allowdisplaybreaks[1]
\usepackage{enumerate}
\usepackage{stmaryrd}
\usepackage[margin=1in]{geometry}

\usepackage[shortlabels]{enumitem}
\crefformat{enumi}{#2#1#3}
\crefrangeformat{enumi}{#3#1#4 to~#5#2#6}
\crefmultiformat{enumi}{#2#1#3}
{ and~#2#1#3}{, #2#1#3}{ and~#2#1#3}

\DeclareSymbolFont{symbolsC}{U}{pxsyc}{m}{n}
\SetSymbolFont{symbolsC}{bold}{U}{pxsyc}{bx}{n}
\DeclareFontSubstitution{U}{pxsyc}{m}{n}
\DeclareMathSymbol{\medcircle}{\mathbin}{symbolsC}{7}

\crefname{algocf}{Algorithm}{Algorithms}

\crefname{equation}{}{} 
\crefname{enumi}{}{} 
\AtBeginEnvironment{appendices}{\crefalias{section}{appendix}} 

\usepackage[color,final]{showkeys} 

\colorlet{refkey}{orange!20}
\colorlet{labelkey}{blue!30}

\crefname{algocf}{Algorithm}{Algorithms}

\numberwithin{equation}{section}
\newtheorem{theorem}{Theorem}[section]
\newtheorem{proposition}[theorem]{Proposition}
\newtheorem{lemma}[theorem]{Lemma}

\crefname{claim}{Claim}{Claims}

\newtheorem{corollary}[theorem]{Corollary}

\newtheorem*{question*}{Question}
\newtheorem{fact}[theorem]{Fact}

\theoremstyle{definition}
\newtheorem{definition}[theorem]{Definition}

\newtheorem*{definition*}{Definition}

\theoremstyle{remark}
\newtheorem*{remark}{Remark}

\newcommand{\snorm}[1]{\lVert#1\rVert}
\newcommand{\sang}[1]{\langle #1 \rangle}

\newcommand{\mb}{\mathbb}

\newcommand{\mbm}{\mathbbm}
\newcommand{\mc}{\mathcal}

\newcommand{\mr}{\mathrm}

\newcommand{\on}{\operatorname}

\newcommand{\wt}{\widetilde}

\newcommand{\eps}{\varepsilon}

\let\originalleft\left
\let\originalright\right
\renewcommand{\left}{\mathopen{}\mathclose\bgroup\originalleft}
\renewcommand{\right}{\aftergroup\egroup\originalright}

\allowdisplaybreaks

\title{The intransitive dice kernel: $\frac{\mathbbm{1}_{x\ge y}-\mathbbm{1}_{x\le y}}{4} - \frac{3(x-y)(1+xy)}{8}$}

\author[A1]{Ashwin Sah}
\author[A2]{Mehtaab Sawhney}
\address{Department of Mathematics, Massachusetts Institute of Technology, Cambridge, MA 02139, USA}
\email{\{asah,msawhney\}@mit.edu}

\thanks{Sah and Sawhney were supported by NSF Graduate Research Fellowship Program DGE-1745302. Sah was supported by the PD Soros Fellowship. Sawhney was supported by the Churchill Foundation.}

\begin{document}
\maketitle
\begin{abstract}
Answering a pair of questions of Conrey, Gabbard, Grant, Liu, and Morrison, we prove that a triplet of dice drawn from the \emph{multiset model} are intransitive with probability $1/4+o(1)$ and the probability a random pair of dice tie tends toward $\alpha n^{-1}$ for an explicitly defined constant $\alpha$. This extends and sharpens the recent results of Polymath regarding the \emph{balanced sequence model}. We further show the distribution of larger tournaments converges to a universal tournamenton in both models. This limit naturally arises from the discrete spectrum of a certain skew-symmetric operator (given by the kernel in the title acting on $L^2([-1,1])$). The limit exhibits a degree of symmetry and can be used to prove that, for instance, the limiting probability that $A_i$ beats $A_{i+1}$ for $1\le i\le 4$ and that $A_5$ beats $A_1$ is $1/32+o(1)$. Furthermore, the limiting tournamenton has range contained in the discrete set $\{0,1\}$. This proves that the associated tournamenton is non-quasirandom in a dramatic fashion, vastly extending work of Cornacchia and H{\k{a}}z{\l}a regarding the continuous analogue of the balanced sequence model.

The proof is based on a reduction to conditional central limit theorems (related to work of Polymath), the use of a ``Poissonization'' style method to reduce to computations with independent random variables, and the systematic use of switching-based arguments to extract cancellation in Fourier estimates when establishing local limit-type estimates.
\end{abstract}

\section{Introduction}\label{sec:introduction}
We consider the following pair of models of random dice.
\begin{definition}\label{def:model}
A $n$-sided die is a sequence of numbers $(a_1,\ldots,a_n)\in [n]^{n}$ such that $\sum_{j=1}^na_j = n(n+1)/2$. In the \emph{multiset model}, the faces of a die $(a_1,\ldots,a_n)$ are sampled as a uniform random nondecreasing sequence in $[n]$ which satisfy $\sum_{j=1}^na_j = n(n+1)/2$. In the \emph{balanced sequence model} the faces of a die $(a_1,\ldots,a_n)$ are sampled as a uniform random sequence in $[n]$ such that $\sum_{j=1}^na_j = n(n+1)/2$.
\end{definition}

We also require a notion of when one die is said to ``beat'' another die. 
\begin{definition}\label{def:beats}
An $n$-sided die $(a_1,\ldots,a_n)$ \emph{beats} another die $(b_1,\ldots,b_n)$ if
\[\sum_{j=1}^n\sum_{k=1}^n\bigg(\mbm{1}_{a_j>b_k} + \frac{1}{2}\mbm{1}_{a_j=b_k}\bigg)>\frac{n^2}{2}.\]
Furthermore we say that die $(a_1,\ldots,a_n)$ \emph{ties} die $(b_1,\ldots,b_n)$ if 
\[\sum_{j=1}^n\sum_{k=1}^n\bigg(\mbm{1}_{a_j>b_k} + \frac{1}{2}\mbm{1}_{a_j=b_k}\bigg)=\frac{n^2}{2}.\]
\end{definition}

Our goal is to study \emph{dice tournaments}. Specifically, we sample $m$ independent random $n$-sided dice, either all from the multiset model or all from the balanced sequence model, and consider the outcome of each pair. We will think of $m$ as fixed while $n$ is tending to infinity.

The phenomenon of intransitive dice are exemplified by an example constructed by Efron in the 1960's \cite{Gar70}: consider the dice\footnote{Notice that as stated these dice do not satisfy the precise sum and face side bounds specified in \cref{def:model}.} 
\[A = (0,0,4,4,4,4),~B = (3,3,3,3,3,3),~C=(2,2,2,6,6),~D=(1,1,1,5,5,5).\]
Efron observed that in this example that $A$ beats $B$, $B$ beats $C$, $C$ beats $D$, and $D$ beats $A$: peculiarly, the relation ``beats'' is not transitive. This phenomenon gathered a substantial amount of popular interest \cite{Sc00,Gri17} including appearing in Martin Gardner's column in Scientific American \cite{Gar70}.

Mathematical work until recently had largely been focused on constructing tournaments with various properties \cite{MM67, FT00, AD17, Aki21, AS21, Yak22}; for instance work of Moon and Moser \cite{MM67} established that given any tournament $T$ there exists a set of dice (not necessarily satisfying the sum constraints of \cref{def:model}) which realize this tournament $T$.\footnote{As it turns out, our main results on random intransitive dice can be used to reprove a number of these results; we refer the reader to \cref{prop:construct}.} However, recently there has been significant interest in understanding random models of intransitive dice due to a set of conjectures raised in work of Conrey, Gabbard, Grant, Liu, and Morrison \cite{CGGLM16}.

In the work of Conrey, Gabbard, Grant, Liu, and Morrison \cite{CGGLM16}, the authors considered dice drawn from the multiset model. While a nice model for dice, one may ask why they do not consider the ``most natural'' model of dice where there is no additional condition on the sum. In this case it is straightforward to observe empirically (and can be proven rigorously) that with high probability whether die $A$ beats die $B$ can be determined simply by looking at the sum of the faces of the dice. Conrey, Gabbard, Grant, Liu, and Morrison \cite{CGGLM16} conducted empirical simulations in the multiset model. Based on these experimental results, they conjectured (\cite[Conjectures~1,2,3]{CGGLM16}) that as $n\to\infty$ (a) the probability a pair of dice tie is $o(1)$ (b) for a random triplet of dice $A$, $B$, and $C$ the probability that $A$ beats $B$, $B$ beats $C$, and $C$ beats $A$ is $1/8+o(1)$ and (c) the tournament associated to dice is quasirandom. (Conrey, Gabbard, Grant, Liu, and Morrison \cite{CGGLM16} equivalently formulate (c) in terms of the probability of various $m$-die tournaments.)

The first rigorous progress towards these conjectures was made by Polymath \cite{Pol22}, where they considered $n$-sided die drawn from not the multiset model but from the balanced-sequence model in \cref{def:model}. In this balanced sequence model, Polymath \cite{Pol22} was able to prove both conjectures (a) and (b) by showing that for almost all dice $A$ drawn from the balanced sequence model, approximately half of the dice from the balanced sequence model beat it. However, based on numerical calculations Polymath conjectured that (c) is false (see discussion surrounding \cite[Conjecture~1.3]{Pol22}). This suspicion was later confirmed in a continuous analogue of the balanced sequence model by work of Cornacchia and H{\k{a}}z{\l}a \cite{CH20} where die faces are sampled from $[0,1]$ uniformly at random. They proved this by studying four-cycle counts and proved that there exists a small absolute constant $\eps > 0$ such that the probability that $A$ beats $B$, $B$ beats $C$, $C$ beats $D$, and $D$ beats $A$ for $n$ large is at least $1/16 + \eps$ (higher than if the underlying tournament was quasirandom). Finally, in work of H{\k{a}}z\l a, Mossel, Ross, and Zheng \cite{HMRZ20}, the phenomenon of transitivity was investigated in the context of die faces which are drawn independently at random from a fixed distribution $\rho$ which is continuous. Remarkably, the phenomenon of intransitivity is extremely delicate and under mild conditions on $\rho$ the only distribution exhibiting any form of intransitivity is the uniform distribution. In all the rigorous work regarding probabilistic models of intransitive dice, the use of local central limit theorem type techniques has been crucial and this has been aided by the fact that the underlying faces of the die are independent modulo conditioning on a simple linear relation. We note this is no longer true in the original multiset model of Conrey, Gabbard, Grant, Liu, and Morrison \cite{CGGLM16} and this served as a key obstacle for extending results to the original model. 

Our main result is a complete characterization of the tournament associated with intransitive dice. Our results are sufficiently strong to naturally explain the results of Polymath \cite{Pol22} and Cornacchia and H{\k{a}}z{\l}a \cite{CH20} and point to a number of surprising phenomena which are not immediately obvious numerically.  

In order to state our main result we will require the definition of a certain operator on $L^2([-1,1])$.
\begin{definition}\label{def:operator}
Consider the skew-symmetric kernel $f\colon[-1,1]^2\to\mb{R}$ defined by
\[f(x,y) = \frac{\mbm{1}_{x\ge y}-\mbm{1}_{x\le y}}{4} - \frac{3(x-y)(1+xy)}{8}.\]
Define the operator $\mc{A}\colon L^2([-1,1])\to L^2([-1,1])$ (with the Lebesgue measure) by the map
\[\mc{A}(g) = \int_{-1}^1f(x,y)g(y)dy.\]
Let $\sigma_1\ge\sigma_2\ge\cdots$ denote real numbers so that $\{\pm i\sigma_\ell\colon\ell\ge 1\}$ forms the discrete spectrum of $\mc{A}$.
\end{definition}
\begin{remark}
Since $\mc{A}$ is real skew-symmetric we have that the spectrum is purely imaginary and coming in pairs. Furthermore, based on numerical computation, a closed form solution for $\sigma_j$ appears unlikely.
\end{remark}

Our main result captures the precise probability distribution associated with the dice tournament.

\begin{theorem}\label{thm:main}
Fix $m\ge 2$ and independently sample $n$-sided dice $A_1,\ldots,A_m$, either all from the multiset model or all from the balanced sequence model. Let $G^{(j)}$ for $1\le j\le m$ be infinite vectors of standard Gaussians and for $1\le j<k\le m$ let
\[H_{jk}=\sum_{\ell\ge 1}\sigma_\ell(G_{2\ell-1}^{(j)}G_{2\ell}^{(k)}-G_{2\ell}^{(j)}G_{2\ell-1}^{(k)}).\]
Then for any digraph $D$ on vertices $[m]$,
\[\lim_{n\to\infty}\mb{P}[A_j\emph{ beats }A_k\emph{ for all }jk\in E(D)]=\mb{P}[H_{jk}>0\emph{ for all }jk\in E(D)].\]
\end{theorem}
\begin{remark}
$H_{jk}$ is defined by a convergent sum almost surely due to the bound $\sum_{\ell\ge t}\sigma_\ell^2=O(1/t)$, which we prove in \cref{lem:operator-decay} (\cref{M7}), and an application of Borel--Cantelli to the random events $\mc{E}_t$ defined by $|\sum_{t\le\ell<2t}\sigma_\ell(G_{2\ell-1}^{(j)}G_{2\ell}^{(k)}-G_{2\ell}^{(j)}G_{2\ell-1}^{(k)})|\ge t^{-1/4}$ for $t$ ranging over powers of $2$. Indeed, $\mb{P}[\mc{E}_t]=O(t^{-1/2})$ by the Chebyshev inequality, which has finite sum over powers of $2$, so all but finitely many $\mc{E}_t$ hold and the convergence follows. Alternatively, we can interpret each individual $H_{jk}$ as a Gaussian with random variance equal to the inverse of an almost surely convergent weighted sum of chi-squared distributions.
\end{remark}
\begin{remark}
The proof of \cref{thm:main} actually shows something stronger, which is that $(H_{jk})_{1\le j<k\le m}$ is the limiting distribution of $(c\cdot\mr{margin}_{jk}/n)_{1\le j<k\le m}$, where $c=1/2$ for the multiset model and $c=1$ for the balanced sequence model and where $\mr{margin}_{jk}$ is by how much die $A_j$ beats $A_k$ (i.e., how many more pairs than $n^2/2$, possibly negative, $A_j$ beats $A_k$ for).
\end{remark}

We note that the statement of \cref{thm:main} may appear slightly strange and difficult to work with; however, a number combinatorial consequences follow in a routine manner given \cref{thm:main}.
\begin{corollary}\label{cor:symmetry}
Sample $m$ independent random $n$-sided dice $A_1,\ldots, A_m$ either all from the multiset model or all from the balanced sequence model. Then for any digraph $D$ on vertices $[m]$ let $D_v$ denote the digraph where all edges emanating from the vertex $v\in[m]$ are reversed. We have
\[\lim_{n\to\infty}\mb{P}[A_j\emph{ beats }A_k\emph{ for all }jk\in E(D)]=\lim_{n\to\infty}\mb{P}[A_j\emph{ beats }A_k\emph{ for all }jk\in E(D_v)]\]
for all $v\in [m]$. Furthermore let $D'$ denote the digraph where all the edges of $D$ are reversed. We have 
\[\mb{P}[A_j\emph{ beats }A_k\emph{ for all }jk\in E(D)]=\mb{P}[A_j\emph{ beats }A_k\emph{ for all }jk\in E(D')].\]  
\end{corollary}

From \cref{thm:main} we see that the probability a pair of dice tie is $o(1)$. Then, considering $D$ to be a directed cycle on $3$ vertices and comparing to $D_v$, along with using permutation symmetry, we immediately see that all labelled $3$-vertex tournaments appear asymptotically with the same probability. Thus a random triplet of dice is intransitive with probability $1/4+o(1)$. This immediately implies the conjectures of Conrey, Gabbard, Grant, Liu, and Morrison \cite[Conjectures~1,~2]{CGGLM16} (and recovers the results of Polymath which proved these two facts in the balanced sequence model).

We can deduce that a forest with $e$ edges occurs with probability $2^{-e}+o(1)$ by iteratively applying \cref{cor:symmetry} to a leaf (and using that ties occur negligibly). We can also deduce that any orientation of a labeled $(2k+1)$-cycle occurs with the same limiting probability $2^{-(2k+1)}+o(1)$ by repeatedly applying the two operations specified in \cref{cor:symmetry}. These are perhaps surprising given the results of Cornacchia and H{\k{a}}z{\l}a \cite{CH20} showing a lack of quasirandomness for continuous dice models. We conjecture, however, that the only equalities between complete tournaments in the limit can be achieved via these symmetries and permutation symmetry. 

For our next corollary, we will require the tournament analogue of a graphon. We refer the reader to \cite[Chapter~4]{Zha23} for a more extensive discussion of graphons.
\begin{definition}\label{def:tournamenton}
Given two measurable functions $U,W\colon[0,1]^2\to\mb{R}$, define the \emph{cut metric} as
\[\delta_\Box(U,W) = \inf_{\phi}\sup_{S,T\in [0,1]}\bigg|\int_{S\times T}U(x,y) - W(\phi(x),\phi(y))dxdy\bigg|\]
where the infimum $\phi$ is taken over all invertible measure preserving maps. We define the \emph{tournamentons} $T_0$ to be the space of all functions $T\colon[0,1]^2\to[0,1]$ such that $T(x,y) = 1-T(y,x)$ and let $\wt{T_0}$ denote the space of tournamentons modulo identifying tournamentons with cut distance $0$.
\end{definition}

As is standard one can identify a graph $G$ with an associated graphon, and similar for a tournamenton, by embedding the adjacency matrix into $[0,1]^2$ (for the tournamenton this requires putting values of $1/2$ on the diagonal); we will carry this transformation out without comment.
\begin{corollary}\label{cor:convergence}
Consider the graph $T_n$ where the vertex set is either (a) all nondecreasing sequences $(a_1,\ldots,a_n)$ in $[n]^n$ such that $\sum_{j=1}^na_j = n(n+1)/2$ or (b) all sequences $(a_1,\ldots,a_n)$ in $[n]^n$ such that $\sum_{j=1}^na_j = n(n+1)/2$, and where there is a directed edge from one sequence to another if the corresponding die beats the other.

Then $T_n$ converges under the cut metric to a tournamenton $\mc{T}$ (which is the same in cases (a) and (b)). Furthermore, the preimage of the set $\{0,1\}$ under $\mc{T}$ has measure $1$.
\end{corollary}
\begin{remark}
Technically $T_n$ may not be a tournament but a partial tournament due to ties, so \emph{a priori} we only have convergence to a \emph{partial tournamenton}; however, a consequence of \cref{thm:main} discussed above is that ties occur with probability $o(1)$ so we will obtain a genuine tournamenton in the limit.
\end{remark}

Note that the density of digraph $D$ in the tournament $T_n$ is precisely the probability that the associated digraph of dice beating other dice occurs when sampling from either the multiset model (case (a)) or the balanced sequence model (case (b)). Thus the density of digraph $D$ in the limit tournament $\mc{T}$ is the limiting probability described by \cref{thm:main}.

The claim that the preimage of the set $\{0,1\}$ has measure $1$ is equivalent to the fact that for every $\eps>0$ there is a $k$ such that a $\mc{T}$-random tournament (defined analogously to a $W$-random graph \cite[Section~4.4]{Zha23}) on $k$ vertices lies in a set of size $2^{\eps k^2}$ with at probability at least $1-\eps$. This equivalence is detailed in \cref{lem:measure-theory}; we will prove \cref{cor:convergence} through this equivalence and prove that one can take a polynomial relation between $k$ and $\eps$. The fact that $\mc{T}\neq 1/2$ corresponds to a lack of quasirandomness. We also establish that the directed $4$-cycle in particular occurs with limiting probability greater than $1/16$ in \cref{prop:4-cycle}, and show that all digraphs $D$ have positive density in $\mc{T}$ in \cref{prop:construct}.

Finally, we also precisely quantify the probability that a given pair of dice are tied beyond the $o(1)$ guaranteed as a consequence of \cref{thm:main}.
\begin{theorem}\label{thm:ties}
Let $A$ and $B$ be dice which are jointly drawn independently from the multiset model. Let $\alpha = 2^{-5/2}\pi^{-1/2}\mb{E}[(\sum_{\ell\ge 1}\sigma_\ell^2(Z_\ell^2+Z_\ell'^2))^{-1/2}]$ where $Z_\ell,Z_\ell'\sim\mc{N}(0,1)$. We have
\[\mb{P}[A\emph{ ties }B]=(\alpha+o(1))n^{-1}\]
for some absolute constant $c=c_{\ref{thm:ties}}>0$. If instead $A$ and $B$ are jointly drawn independently from the balanced sequence model then
\[\mb{P}[A\emph{ ties }B]=(2\alpha+o(1))n^{-1}.\]
\end{theorem}
\begin{remark}
This can be heuristically reconstructed by considering the second remark after \cref{thm:main} with $m=2$. $H_{12}$ is the limiting distribution of $c\cdot\mr{margin}_{12}/n$ (where $c=1/2$ for the multiset model and $c=1$ for the balanced sequence model). If we imagine that the mass of this distribution was discretized in the obvious way along all possible values of $\mr{margin}_{12}$ in the lattice $\mb{Z}/2$, we obtain the above. In fact, one can use the techniques in \cref{sec:ties-proof} to show a local limit theorem for $\mr{margin}_{12}$:
\[\mb{P}[\mr{margin}_{12}=x]=\frac{c}{2n}f_{H_{12}}(cx/n)+o(1/n)\]
uniformly for $x\in\mb{Z}/2$ where $f_{H_{12}}$ is the probability density function of $H_{12}$. We do not prove this here since the technical details are quite involved, but note that \cref{thm:ties} is the $x=0$ case.
\end{remark}

We interpret the constant $\alpha$ as the (inverse) standard deviation around the best linear approximant (in the sense of Ordinary Least Squares) to a conditioned Brownian motion at the end of \cref{sec:ties-proof}.

In general, a tournament $T$ with exactly $t$ ties among $m$ dice, and $\binom{m}{2}-t$ prescribed outcomes of the other match-ups, where $m$ and $0\le t\le\binom{m}{2}$ are fixed, should occur with probability $(c_T+o(1))n^{-t}$. We do not pursue such a general statement here though similar techniques may apply and a probabilistic interpretation of the constant $c_T$ should arise from \cref{thm:main} similar to the case $(m,t)=(2,1)$ above.

\subsection{First steps, proof outline, and organization}\label{sub:organ}
Our techniques at a high level involve Fourier analysis in the style of local limit theorems. In particular, we study various ``conditional Fourier coefficients'' in detail to show that the normalized joint distribution of ``victory margins'' (see the second remark following \cref{thm:main}) converges to $(H_{jk})_{1\le j<k\le m}$. We also use a more detailed analysis involving additional control on the ``coarseness'' of certain modified statistics of random dice to get very good local control of the event that there is a precise tie. We defer a more detailed proof outline to \cref{sec:sketch} after developing the basic tools to attack the problem in \cref{sec:count}.

The first step in proving \cref{thm:main} (in the multiset model) relies on observing that while the dice face in the multiset model are nonindependent, the frequency statistics can be given a natural ``near-independent'' model. This ultimately relies on a well-known bijection between the multiset model and the simple random walk; the details appear in \cref{lem:crucial}. We note that in the context of the balanced sequence model, \cref{lem:crucial} reduces to the ``Poissonization'' trick. Given this we interpret the ``beats'' relation through frequency counts (\cref{lem:beats,lem:beats-matrix}) and the operator in \cref{def:operator} arises naturally. These initial steps are carried out in \cref{sec:count}, and provide the key starting point to understand the necessary distributions from a Fourier perspective.

Given the setup in \cref{sec:count}, we provide a heuristic outline of the argument for \cref{thm:main,thm:ties} and an overview of the various consequences in \cref{sec:sketch}. We then collect a list of technical preliminaries which will be used throughout the paper in \cref{sec:prelim}. We prove various Fourier coefficient bounds used in the proofs of \cref{thm:main,thm:ties} in \cref{sec:fourier}. We prove \cref{thm:main} in \cref{sec:translate-fourier}, modulo a technical ingredient proven in \cref{sec:coefficients}, and then collect various consequences following from \cref{thm:main} in \cref{sec:consequence}. Finally we prove \cref{thm:ties} in \cref{sec:ties-proof}.

\subsection*{Notation}
We write $f=O(g)$ to mean that $f\le Cg$ for some absolute constant $C$, and $g=\Omega(f)$ and $f\lesssim g$ to mean the same. We write $f=o(g)$ if for all $c > 0$ we have $f\le cg$ once the implicit growing parameter (typically $n$) grows large enough, and $g=\omega(f)$ means the same. Subscripts imply a dependence of these implicit constants on those parameters. We use $\overset{d.}{=},\overset{d.}{\rightarrow}$ for distributional equality and limits, respectively.

For $\mu\in\mb{R}^d$ and positive semidefinite $\Sigma\in\mb{R}^{d\times d}$ we let $\mc{N}(\mu,\Sigma)$ be the Gaussian vector with mean $\mu$ and covariance matrix $\Sigma$. For finite matrices $M$ we will use $M_{ij}$ to denote the entry in the $(i,j)$ position. Throughout this paper all logarithms are base $e$.

\subsection*{Acknowledgments}
We thank Timothy Gowers, Michael Ren, and Mark Sellke for useful comments and discussions.

\section{Count statistics of balanced sequence model and multiset model}\label{sec:count}
The idea to get a handle on the multiset model is to create a procedure for sampling which derives from a sequence of independent random variables. We will require the notion of a frequency statistic which will be crucial for our purposes.
\begin{definition}\label{def:freq}
Given an $n$-sided die $A = (a_1,\ldots,a_n)$ define the \emph{frequency counts} of $A$ to be
\[\wt{a}_i = |\{j\colon a_j = i\}|\]
for $1\le i\le n$.
\end{definition}

The key point is the following distributional claim regarding the frequency counts of a die drawn from either multiset or balanced sequence model, which relates these models to a sequence of either geometric (in the multiset case) or a sequence of Poisson random variables (in the balanced sequence case). In the balanced sequence case this is essentially equivalent to the ``Poissonization'' trick.
\begin{lemma}\label{lem:crucial}
We have the following:
\begin{itemize}
    \item If $B$ is drawn from the multiset model we have 
    \[(\wt{b}_1,\ldots,\wt{b}_n) \overset{d.}= (G_1,\ldots,G_n)\]
    where $G_j$ are sampled as follows: draw independent $\mr{Geom}(1/2)$\footnote{Here $X\overset{d.}{=}\mr{Geom}(1/2)$ means $\mb{P}[X = k] = (1/2)^{k+1}$ for $k\in\{0,1,\ldots\}$. Note this is $0$-indexed, corresponding to the number of ``failures'' before a repeatedly flipped fair coin shows heads, instead of the number of ``trials''.} random variables $G_i$ and then condition on $\sum_{j=1}^nG_j = n$ and $\sum_{j=1}^njG_j = n(n+1)/2$.
    \item If $B$ is drawn from the balanced sequence model we have 
    \[(\wt{b}_1,\ldots,\wt{b}_n) \overset{d.}= (P_1,\ldots,P_n)\]
    where $P_j$ are sampled as follows: draw independent $\mr{Pois}(1)$ random variables $P_j$ and then condition on $\sum_{j=1}^nP_j = n$ and $\sum_{j=1}^njP_j = n(n+1)/2$.
\end{itemize}
\end{lemma}
\begin{proof}
We consider the first case. Notice that the multiset model of a die can equivalently be sampled by sampling a uniformly random right-up walk between $(1,1)$ and $(n+1,n)$ and looking at the height of each rightward step, conditional on the area under the walk being $n(n+1)/2$. Indeed, there is a standard bijection between nondecreasing integer sequences $(b_1,\ldots,b_n)$ with $1\le b_j\le n$ and such walks: each $a_j$ corresponds to a rightward step from $(j,b_j)$ to $(j+1,b_j)$; furthermore, the area under the walk ends up being $b_1+\cdots+b_n$. Notice that drawing such a walk is equivalent to looking at an infinite random walk which takes steps in the directions $(1,0)$ or $(0,1)$ each with probability $1/2$ and then conditioning on starting at $(1,1)$ and passing through both $(n+1,n)$ and $(n+1,n+1)$, then truncating appropriately.

Now define $G_j$ as precisely the length of the horizontal segment on the line $y=j$ in this conditioned infinite random walk. In the unconditioned infinite walk, we have that the lengths (which might be $0$) of these horizontal segments in order have an independent distribution where the law is by definition a sequence of independent geometric random variables with parameter $1/2$. Notice that conditioning on the walk passing through the line segment $(n+1,n)$ and $(n+1,n+1)$ guarantees that $\sum_{j=1}^nG_j = n$ and the conditioning on area corresponds exactly to $\sum_{j=1}^njG_j = n(n+1)/2$. Considering the bijection defined above, these conditioned $G_j$ then correspond directly to the $\wt{b}_j$.

The second case is rather simpler. Notice that if one draws $n$ faces from $[n]$ uniformly at random then we have the proportionality
\[\mb{P}[(\wt{b}_1,\ldots,\wt{b}_n) = (b_1,\ldots,b_n)]\propto\prod_{j=1}^n\frac{1}{b_j!}\]
for tuples $(\wt{b}_j)_{1\le j\le n}\in\{0,\ldots,n\}^n$ with sum $n$. The result then follows since $\mb{P}[\mr{Pois}(1)=k]=\frac{e^{-1}}{k!}$ for $k\in \mb{Z}$, and since conditioning on the sum of the dice being $n(n+1)/2$ corresponds to conditioning on $\sum_{j=1}^njP_j=n(n+1)/2$.
\end{proof}

The precise reason this description is useful is that given a die $B$ one can define a linear function of the frequency count statistics of another die $A$ which captures precisely whether $A$ beats $B$ or not. An equivalent computation appears in the work of Polymath \cite[Section~4]{Pol22}; the formulation presented there however is more naturally a linear function of the ``die faces'' instead of the ``frequency count statistics''.
\begin{lemma}\label{lem:beats}
 We have that a die $A$ with sides $(a_1,\ldots,a_n)$ beats a die $B$ with sides $(b_1,\ldots,b_n)$ if and only if
 \[\sum_{j=1}^n\bigg(\sum_{1\le k<j}\wt{b}_k + \frac{\wt{b}_j}{2} - (j-1/2)\bigg)\wt{a}_j>0\]
 and ties if and only if the sum on the left is $0$.
\end{lemma}
\begin{proof}
Notice that 
\begin{align*}
\sum_{i=1}^n\sum_{j=1}^n\bigg(\mbm{1}_{a_i>b_j} + \frac{1}{2}\mbm{1}_{a_i = b_j}\bigg) &= \sum_{j=1}^n\wt{a}_j\sum_{1\le k<j}\wt{b}_j + \sum_{j=1}^n \frac{\wt{a}_j\wt{b}_j}{2}\\
&= \sum_{j=1}^n\wt{a}_j\bigg(\sum_{1\le k<j}\wt{b}_j + \frac{\wt{b}_j}{2}\bigg).
\end{align*}
Since $A$ is an $n$-sided die with sum of faces $n(n+1)/2$ we have 
\[\sum_{j=1}^n\wt{a}_j(j-1/2) = \sum_{j=1}^na_j - \frac{n}{2} = \frac{n^2}{2}\]
and therefore the event that $A$ beats $B$ is precisely equivalent to 
\[\sum_{j=1}^n\wt{a}_j\bigg(\sum_{1\le k<j}\wt{b}_j + \frac{\wt{b}_j}{2} - (j-1/2)\bigg) > 0\]
whereas $A$ and $B$ being tied corresponds to the left side being equal to $0$. 
\end{proof}

We cast this condition in an equivalent form which will be useful for computations involving Gaussians.
\begin{definition}\label{def:dice-matrix}
Let $I_n$ be the $n\times n$ identity matrix. Let $\vec{v}_1,\vec{v}_2\in\mb{R}^n$ be defined by $v_{1i}=1/\sqrt{n}$ for $1\le i\le n$ and $v_{2i}=(i-(n+1)/2)/\sqrt{n(n^2-1)/12}$ for $1\le i\le n$. Note these are orthogonal unit vectors. Let $M_n\in\mb{R}^{n\times n}$ be defined via $(M_n)_{ij}=\mbm{1}_{i<j}+(\mbm{1}_{i=j}/2)$ and $M_n^\ast\in\mb{R}^{n\times n}$ via
\[M_n^\ast=(I_n-\vec{v}_2\vec{v}_2^T)(I_n-\vec{v}_1\vec{v}_1^T)M_n(I_n-\vec{v}_1\vec{v}_1^T)(I_n-\vec{v}_2\vec{v}_2^T)=(I_n-\vec{v}_1\vec{v}_1^T-\vec{v}_2\vec{v}_2^T)M_n(I_n-\vec{v}_1\vec{v}_1^T-\vec{v}_2\vec{v}_2^T).\]
Equivalently, we re-express the (asymmetric) bilinear form $M_n$ in a basis including $\vec{v}_1,\vec{v}_2$ on both sides, zero out the rows and columns corresponding to $\vec{v}_1,\vec{v}_2$, and then convert back. Finally, define $\sigma_{n,1}\ge\cdots\ge\sigma_{n,\lfloor n/2\rfloor}$ be such that $\{\pm i\sigma_{n,\ell}\colon\ell\in[\lfloor n/2\rfloor]\}$ is the spectrum (or the spectrum minus a copy of $0$ if $n$ is odd).
\end{definition}

The following lemma introduces this discrete variant of the kernel which appears in the title of the paper.
\begin{lemma}\label{lem:beats-matrix}
Given dice $A,B$ with frequency vectors $\vec{a},\vec{b}\in\{0,\ldots,n\}^n$, we have that $A$ beats $B$ if and only if
\[\wt{b}^TM_n^\ast\wt{a}>0.\]
\end{lemma}
\begin{proof}
This is immediate since a simple manipulation of \cref{lem:beats} shows the condition that $A$ beats $B$ is equivalent to $(\vec{b}-\sqrt{n}\vec{v}_1)^TM_n(\vec{a}-\sqrt{n}\vec{v}_1)>0$, and since $\vec{v}_1^T(\wt{a}-\sqrt{n}\vec{v}_1) = \vec{v}_2^T(\wt{a}-\sqrt{n}\vec{v}_1)=\vec{v}_1^T(\wt{b}-\sqrt{n}\vec{v}_1) = \vec{v}_2^T(\wt{b}-\sqrt{n}\vec{v}_1) = 0$.
\end{proof}

Finally, we record some properties of $M_n^\ast$ as well as $\mc{A}$ (\cref{def:operator}). We are brief with the details as it mostly amounts calculation with explicit functions and operators.
\begin{lemma}\label{lem:operator-decay}
There exists $C = C_{\ref{lem:operator-decay}}>0$ such that the following holds. Let $M_n^\ast$ be as in \cref{def:dice-matrix}, $x = (n + 1 - 2i)/(n-1)$ and $y = (n + 1-2j)/(n-1)$. Then we have the following:
\begin{enumerate}[{\bfseries{M\arabic{enumi}}}]
    \item\label{M1} \[(M_n^\ast)_{ij} = \frac{\mbm{1}_{x\ge y}-\mbm{1}_{x\le y}}{2} - \frac{3(x-y)(1-1/n)}{4} - \frac{3xy(x-y)(n-1)^2}{4n(n+1)}.\]
    \item\label{M2} $M_n^\ast$ is skew-symmetric.
    \item\label{M3} $\snorm{M_n^\ast}_{1\to\infty}\le C_{\ref{lem:operator-decay}}$ (i.e., the entries are of bounded size).
    \item\label{M4} $\snorm{M_n^\ast}_{1\to2}=\snorm{M_n^{\ast T}}_{1\to 2}\le C_{\ref{lem:operator-decay}}\sqrt{n}$ (i.e., the row and column $L^2$-norms are $O(\sqrt{n})$ in size).
    \item\label{M5} $\snorm{M_n^\ast}_F/n\in [C_{\ref{lem:operator-decay}}^{-1},C_{\ref{lem:operator-decay}}]$\footnote{Here the \emph{Frobenius norm} of matrix $M\in\mb{R}^{n\times n}$ is $\snorm{M}_F:=\sqrt{\sum_{1\le j,k\le n}M_{jk}^2}$}.
    \item\label{M6} For fixed $t\ge 1$ and $n$ sufficiently large,
    \[\sum_{\ell\ge t}\sigma_{n,\ell}^2\le C_{\ref{lem:operator-decay}}n^2/t.\]
    \item\label{M7} For $1\le i,j,k\le n$ we have $|(M_n^\ast)_{ji}-(M_n^\ast)_{ki}|\le C_{\ref{lem:operator-decay}}|j-k|/n$ for all $i\notin [j,k]\cup[k,j]$.
\end{enumerate}
We also have the following properties of $\mc{A}$.
\begin{enumerate}[{\bfseries{M\arabic{enumi}}}]
    \setcounter{enumi}{7}
    \item\label{M8} For all $t\ge 1$ we have that $t\sigma_t\in [C_{\ref{lem:operator-decay}}^{-1},C_{\ref{lem:operator-decay}}]$ (i.e., $t\sigma_t$ is bounded above and below by an absolute constant).
    \item\label{M9} $(\sigma_{n,\ell}/n)_{1\le\ell\le t}\to(\sigma_\ell)_{1\le\ell\le t}$ as $n\to\infty$.
\end{enumerate}
\end{lemma}
\begin{proof}
Via a direct, albeit tedious computation, one has that if $x = (n + 1 - 2i)/(n-1)$ and $y = (n + 1-2j)/(n-1)$ then
\[(M_n^\ast)_{ij} = \frac{\mbm{1}_{x\ge y}-\mbm{1}_{x\le y}}{2} - \frac{3(x-y)(1-1/n)}{4} - \frac{3xy(x-y)(n-1)^2}{4n(n+1)}.\]
The properties \cref{M2,M3,M4,M5,M7} all follow immediately via direct inspection.

To prove \cref{M6}, i.e. that $\sum_{\ell\ge t}\sigma_{n,\ell}^2$, it suffices to show that there is a rank $t+4$ (say) approximation of $M_n^\ast$, call it $R_t$, such that $\snorm{M_n^\ast-R_t}_F^2\lesssim n^2/t$. This follows from considering a rank $t$ approximation for $M_n$ and then plugging it into \cref{def:dice-matrix}. An appropriate rank $t$ approximation for $M_n$ with square-error $O(n^2/t)$ can be formed by removing square matrices of $1$s from the right isosceles triangle above the main diagonal of $M_n$ in a dyadic fashion.

To prove the convergence given in implied in \cref{M8}, we proceed by an argument identifying matrices with operators in $L^2([-1,1])$ via step functions. In particular, consider the matrices $M_n^\ast$ and identify them with the kernels
\[M_n^{(\ast)}(x,y) = \frac{nM_n^\ast(\lceil n(1-x)/2\rceil,\lceil n(1-y)/2\rceil)}{2}\]
and note that the action of $M_n^\ast$ on $\mb{R}^n$ corresponds exactly to the action of kernel $M_n^{(\ast)}(x,y)$ on step functions where the index $i\in[n]$ has been mapped to the interval $[1-2i/n,1-2(i-1)/n)$. These have the same spectrum: the multiplicative factor of $n/2$ corresponds to fact that the step function which is $1$ on a single length $2/n$ interval has norm $(2/n)^{1/2}$ in the continuous formulation while it has norm $1$ when viewed as a vector in $\mb{R}^n$.

In general, given a kernel $K\colon[-1,1]^2\to\mb{R}$ one can define the integral operator 
\[\wt{K}\colon g(x)\to\int_{-1}^1K(x,y)g(y)dy\]
and we have $\snorm{\wt{K}}_{L^2([-1,1])\to L^2([-1,1])}\le\snorm{K}_{L^2([-1,1]^2)}$ by Cauchy--Schwarz (see e.g.~\cite[Example~9.23]{HN01}). Via this identification, we have the strong convergence $M_n^{(\ast)}(x,y)/n\to\wt{M^{(\ast)}}:=\mc{A}$ where the corresponding kernel is
\[f(x,y)=\frac{\mbm{1}_{x\ge y}-\mbm{1}_{x\le y}}{4} - \frac{3(x-y)(1+xy)}{8}.\]
Given this, since $\mc{A},M_n^{(\ast)},M_n^\ast$ are skew-symmetric (hence normal) operators, it is easy to see that the normalized eigenvalues of $M_n^\ast$ converge to those specified by \cref{def:operator} (as strong convergence implies convergence of the spectrum). This proves \cref{M8}.

Finally we prove \cref{M7}. In order to prove \cref{M7}, we first note that $\mc{A}$ is an $O(1)$-rank skew-symmetric perturbation of the integral operator associated to the function $g(x,y) = \frac{\mbm{1}_{x\ge y}-\mbm{1}_{x\le y}}{4}$. We claim that it suffices to prove that the $t$th singular value of $\wt{g}$ scales as $\Theta(1/t)$. Indeed, apply the generalized Weyl's inequality to the Hermitian operator $\wt{g}^\dagger\wt{g}$ using that $\mc{A}^\dagger\mc{A}$ is a bounded rank perturbation.

To compute the spectrum of $\wt{g}$ (and thus that of $\wt{g}^\dagger\wt{g}$), note that the matrix given by $(T_n)_{ij} = \mbm{1}_{i\ge j} - \mbm{1}_{i\le j}$ has characteristic polynomial $\frac{(-1)^n((\lambda+1)^n+(\lambda-1)^n)}{2}$; this is easily proven via row operations and induction. It follows that the eigenvalues of $T_n$ are $(1+\exp(\pi i (2j-1)/n))/(1-\exp(\pi i (2j-1)/n))$ for $1\le j\le n$. Thus the $j$th largest eigenvalue in magnitude scales as $\Theta(n/j)$. The desired result then follows by rescaling and taking $n\to\infty$.
\end{proof}

\section{Outline of the remainder of the proof}\label{sec:sketch}
We now outline the remainder of the proofs of \cref{thm:main,thm:ties} in the multiset model; the balanced sequence model is very similar modulo adjusting various constant factors arising due to $\mr{Var}[\mr{Geom}(1/2)] = 2\mr{Var}[\mr{Pois}(1)]=2$. We also discuss the various deductions which follow from \cref{thm:main}. Consider a set of $m$ dice $A_1,\ldots,A_m$ and let $\wt{a}_k=(\wt{a}_{kj})_{1\le j\le n}$ be the $n$-dimensional vector corresponding to the frequency counts of $A_k$ for $1\le k\le m$.

\subsection{\texorpdfstring{\cref{thm:main}}{Theorem 1.4} and its consequences}\label{sub:main-proof}
By \cref{lem:beats-matrix} we have that $A_j$ beats $A_k$ if an only if $\wt{x}_k^TM_n^\ast\wt{x}_j>0$. Note that the constraints that $\wt{x}_j$ satisfy are precisely $(1,\ldots,1)^T\wt{x}_j = n$ and $(1,2,\ldots,n)^T\wt{x}_j = n(n+1)/2$ (equivalently, $\wt{x}_j-1$ is orthogonal to $\vec{v}_1,\vec{v}_2$). By construction we have that $M_n^\ast\vec{1} = \vec{0}$ and $M_n^\ast(1,2\ldots,n) = \vec{0}$. Therefore for the sake of reasoning heuristically, we can pretend that the conditioning in \cref{lem:crucial} does not affect the probability distribution of $\wt{x}_k^TM_n^\ast\wt{x}_j$ and instead suppose that $\wt{x}_\ell$ are replaced by $X_\ell$, $n$-dimensional vectors where every entry is taken independently at random to be $\mr{Geom}(1/2)$. Now $X_k^TM_n^\ast X_j$ is a bilinear polynomial of independent random variables. Tools such as the invariance principle of Mossel, O'Donnell, and Oleszkiewicz \cite{MOO10} imply that the associated distribution is close to the distribution in the case where $X_\ell$ are replaced by $Z_\ell$ where each entry of $Z_\ell$ is an independent normal of variance $\mr{Var}[\mr{Geom}(1/2)]=2$. Given this, we can convert to a Gaussian quadratic form. This is invariant under orthogonal transformation, so a singular value decomposition for the skew-symmetric matrix $M_n^\ast$ and an appropriate variant of the spectral theorem quickly leads to the distribution in \cref{thm:main}. In particular, the coefficients associated in \cref{thm:main} arise precisely from an application of \cref{lem:operator-decay}.

In order to prove this heuristic, we need to be precisely understand the joint distribution of $(1,\ldots,1)^TX_j$, $(1,2,\ldots,n)^TX_j$, and the desired quadratic forms. We proceed using Fourier transform (characteristic function) and computing the multidimensional Fourier coefficients of the joint distribution of the quadratic forms conditional on these linear equalities. This conditional expectation can be recast using Bayes' theorem and converted to an expression involving joint coefficients involving both quadratic and linear forms, which we can provide control for using the techniques in \cref{sec:fourier}. Our proof here is closely related to that of that in the work of Polymath \cite{Pol22} which similarly used local central limit theorem techniques to decouple various linear conditions; however the implementation is performed in a rather different manner.

We write this more explicitly. For the sake of this discussion, let $\mc{E}$ denote the event that all that the $m$ samples $X_\ell$ for $1\le\ell\le m$ satisfy $(1,\ldots,1)^TX_j = n,(1,2,\ldots,n)^TX_j = n(n+1)/2$. We then must compute
\[\mb{E}\bigg[\exp\bigg(i\sum_{1\le j<k\le m}\theta_{jk}X_k^TM_n^\ast X_j\bigg)\bigg|\mc{E}\bigg]\]
for all choices of $\theta=(\theta_{jk})_{1\le j<k\le m}$ where $\snorm{\theta}_\infty$ is roughly $\wt{O}(1/n)$.

Via applying Bayes' rule, this amounts to computing 
\[\mb{E}\bigg[\exp\bigg(i\sum_{1\le j<k\le m}\theta_{jk}X_k^TM_n^\ast X_j\bigg)\mbm{1}_{\mc{E}}\bigg],\]
since then considering $\theta=0$ gives an estimate for $\mb{E}[\mbm{1}_{\mc{E}}]=\mb{P}[\mc{E}]$ and we can divide to obtain the conditional expectation. At this juncture, much as in the work of Polymath \cite{Pol22}, we rely on the Fourier inversion formula to convert the indicator $\mbm{1}_{\mc{E}}$ into a explicit integral formula in terms of additional Fourier terms involving the above linear forms. (Note that this conversion is only available to us in the multiset model due to the key lemma \cref{lem:crucial}, and even in the balanced sequence model we utilize the setup of \cref{lem:crucial} to prove \cref{thm:main}.)

In particular, by applying Fourier inversion on the lattices we will find 
\begin{align*}
&\mb{E}\bigg[\exp\bigg(i\sum_{1\le j<k\le m}\theta_{jk}X_k^TM_n^\ast X_j\bigg)\mbm{1}_{\mc{E}}\bigg]\\
&=(2\pi)^{-2m}\int_{[-\pi,\pi]^{2m}}\mb{E}\bigg[\exp\bigg(i\sum_{1\le j<k\le m}\Theta_{k,j}X_k^TM_n^\ast X_j\bigg)\\
&\qquad\qquad\qquad\qquad\qquad\cdot\exp\bigg(i\bigg(\sum_{r=1}^m\xi_{1r}\bigg(\sum_{j=1}^n(X_{rj}-1)\bigg) + \xi_{2r}\bigg(\sum_{j=1}^nj(x_{rj}-1)\bigg)\bigg)\bigg)\bigg]d\vec{\xi}.
\end{align*}
In order to prove the desired result, we split the integral into several regions. If any $|\xi_{1r}|\ge n^{-1/2}(\log n)^7$ or $|\xi_{2r}|\ge n^{-3/2}(\log n)^6$, we prove that the corresponding term in the integral is super-polynomially small using \cref{lem:theta2-high,lem:theta2-mid,lem:theta1-region}. Specifically, we conditions on everything outside of the index $r$, and then the corresponding Fourier integral is simply a product of independent terms handled by these lemmas. To prove these lemmas, we extract cancellation in a systematic and clean manner by considering pairs and triplets of indices and performing ``switches'' between then in order to extract Boolean randomness. These switches allow for one to provide sufficient conditions on various coefficient sequences to be good enough to perform these arguments, and said conditions exist purely in ``physical space'' (whereas the approach taken in the work of Polymath \cite{Pol22} naturally leads one to consider how various coefficients are distributed with respect to angles on the torus). Finally, in the region where $|\xi_{1r}|\le n^{-1/2}(\log n)^{7}$ and $|\xi_{2r}|\le n^{-3/2}(\log n)^{6}$ we apply a Lindeberg exchange argument (see \cite{Lin22}, and also the related proof of the invariance principle \cite{MOO10}) to replace the geometric random variables with Gaussians of the same variance. Using the rapid decay of Fourier coefficients the Gaussian and Gaussian rotational symmetry one can verify the Fourier coefficient matches that of the associated Gaussian prediction and thus the desired result follows via L\'evy continuity and similar techniques which convert Fourier control back to physical space control.

In order to prove \cref{cor:symmetry}, we directly cite \cref{thm:main} and uses symmetries of the Gaussian distribution under negation to derive the necessary result. For \cref{cor:convergence}, note that convergence to a tournamenton follows from general machinery since we have the convergence of each digraph. To deduce that the associated tournamenton is $\{0,1\}$-valued we reduce to proving a random tournament on $M$ dice takes on outcomes within a specific set of complete tournaments of size $2^{\eps M^2}$ with probability at least $1-\eps$. This is shown using \cref{thm:main}: note that we can simulate the limiting tournament on $M$ vertices by sampling the Gaussians $G_\ell^{(j)}$ for $1\le j\le M$ and $\ell\ge 1$ and computing the various $H_{jk}$ and checking their signs. By revealing for each $j$ the first $2M^{1/2}$ Gaussians $Z_\ell^{(j)}$ within a rounding error of $M^{-10}$, this provides at most $\exp(O(M^{3/2}\log M))$ buckets where almost all the probability mass lies and also allows us with good probability to determine the outcome of almost all match-ups in the tournament (this deduction requires Gaussian anticoncentration results such as \cref{thm:gauss-anti} in order to see that it is unlikely that many match-ups are ``too close to call'' due to the rounding error). Then revealing the outcomes of the remaining match-ups introduces $\exp(o(M^2))$ total buckets that contain almost all the probability mass, and which uniquely determine the outcome of the $M$-die tournament.

Given the non-quasirandomness of the associated tournament from \cref{cor:convergence} and the underlying symmetries in \cref{cor:symmetry} it also follows from a simple Cauchy--Schwarz argument that the limiting probability $A$ beats $B$, $B$ beats $C$, $C$ beats $D$, and $D$ beats $A$ is strictly larger than $1/16$ (see \cref{prop:4-cycle}). Finally, we note that via carefully choosing various Gaussians $Z_\ell^{(j)}$ to lie in certain ranges one can prove that the limiting probability of any fixed $M$-die tournament occurring is strictly positive (see \cref{prop:construct}). This allows one to quickly deduce a number of prior results as discussed in the introduction.

\subsection{Proof of \texorpdfstring{\cref{thm:ties}}{Theorem 1.8}}
To compute the probability two dice tie, proceed via a more delicate route. As discussed in the remark following \cref{thm:ties}, one can see this as a (special case of a) local limit theorem version of \cref{thm:main} with two dice.

We use ideas closely related to those in the proof of \cref{thm:main}, as well as additional Fourier coefficient estimates (\cref{lem:theta3-high,lem:theta3-mid}) which use the extra condition that certain associated coefficient sequences ``resemble a simple random walk at all scales'' in a coarse sense. It follows that for almost all outcomes of die $A_1$, the probability a random die $A_2$ with frequency counts $\wt{x}_2$ ties $A_1$ is proportional to $\snorm{M_n^\ast\vec{x}_2}_2^{-1}$. (We note that such a result for the balanced sequence model is essentially implicit in the work of Polymath \cite{Pol22} although not stated in such a manner; however, again, our work proceeds through frequency counts instead of using independent die faces which are not available for the multiset model.)

Therefore the natural approach at this point would be to prove a limit theorem for $\snorm{M_n^{\ast}X}_2^2$, where $X$ is a sequence of geometric random variables conditional on the two linear constraints $(1,\ldots,1)^TX=n,(1,2,\ldots,n)^TX=n(n+1)/2$. While this appears to be possible note that $\snorm{M_n^\ast\vec{x}}_2^2$ is a genuinely quadratic polynomial in the underlying random variables (instead of being multilinear in the case of \cref{thm:main}) and hence for a direct approach various tools developed by Berkowitz \cite{Ber16}, developed in the context of local central limit theorems for clique counts in dense random graphs, would appear to be necessary, which would greatly complicate the situation.

To circumvent this, we proceed indirectly so as to only require linear Fourier estimates. The basic idea is that given a sufficiently good upper bound on  $\snorm{M_n^\ast X}_3$ (conditional on our two linear constraints), by sampling a fixed number of random coordinates $j_1,\ldots,j_T$ for some large constant $T$ we have
\[\snorm{M_n^\ast X}_2^2 \approx \frac{n}{T}\sum_{\ell=1}^T\sang{e_{j_\ell},M_n^\ast X}^2\]
holds with high probability as $T\to \infty$. Therefore the question can be reduced to a question of understanding the linear statistics $(\sang{e_{j_\ell},M_n^{\ast}X})_{1\le \ell\le T}$ jointly conditional on our two linear constraints. This can be handled by precisely the techniques developed we discussed in \cref{sub:main-proof} for \cref{thm:main}. The estimates are necessarily a bit delicate since the function $y\mapsto 1/y$ is not bounded near $0$ and thus care must be taken to rule out the pathology that $\snorm{M_n^\ast X}_2$ is small with unusually large probability.

\section{Preliminaries}\label{sec:prelim}
We briefly collect a series of preliminaries which will be used throughout the proof. First we will require a version of the classical Bernstein inequality, which generalizes Chernoff.
\begin{theorem}[{\cite[Theorem~2.8.1]{Ver18}}]\label{thm:bernstein}
For a random variable $X$ define the $\psi_1$-norm
\[\snorm{X}_{\psi_1}=\inf\{t>0\colon\mb{E}[\exp(|X|/t)]\le 2\}.\]
There is an absolute constant $c = c_{\ref{thm:bernstein}}> 0$ such that the following holds. If $X_1,\ldots,X_N$ are independent random variables then 
\[\mb{P}\bigg[\bigg|\sum_{i=1}^NX_i\bigg|\ge t\bigg]\le 2\exp\bigg(-c\min\bigg(\frac{t^2}{\sum_{i=1}^N\snorm{X_i}_{\psi_1}^2},\frac{t}{\max_i\snorm{X_i}_{\psi_1}}\bigg)\bigg)\]
for all $t\ge 0$.
\end{theorem}

Next we will require the Azuma--Hoeffding inequality (see \cite[Theorem~2.25]{JLR00}).
\begin{lemma}[Azuma--Hoeffding inequality]\label{lem:azuma}
Let $X_0, \ldots, X_n$ form a martingale sequence such that $|X_k-X_{k-1}|\le c_k$ almost surely. Then 
\[\mb{P}[|X_0-X_n|\ge t]\le 2\exp\bigg(-\frac{t^2}{2\sum_{k=1}^nc_k^2}\bigg)\]
\end{lemma}
\begin{remark}
We will refer to $\sum_{k=1}^nc_k^2$ as the \emph{variance proxy} in such a situation.
\end{remark}

Furthermore we will require the Carbery--Wright theorem \cite{CW01} for which prove that low-degree functions of Gaussians are anticoncentrated; we will only require the quadratic case.
\begin{theorem}[{see e.g. \cite[Theorem~1.4]{MNV16}}]\label{thm:gauss-anti}
Fix an integer $d\ge 1$. There exists a constant $C_d$ such that the following holds. For any $\eps>0$, if $(G_i)_{1\le i\le n}$ are independent Gaussian random variables, and $P$ is a polynomial of degree at most $d$ then 
\[\sup_{t\in \mb{R}}\mb{P}\big[|P(G_1,\ldots,G_n)-t|\le\eps\sqrt{\mr{Var}(P(G_1,\ldots,G_n))}\big]\le C_d\eps^{1/d}.\]
\end{theorem}

We will also require the invariance principle of Mossel, O'Donnell, and Oleszkiewicz \cite{MOO10}. The version stated in \cref{thm:invariance} below is a stated as \cite[(11.66)]{O14} (with the necessary hypercontractivity following from \cite[Proposition~3.16]{MOO10}).

\begin{definition}\label{def:influence}
Given a multilinear polynomial $g(x_1,\ldots,x_n) = \sum_{S\subseteq [n]} a_S\prod_{i\in S}x_i$, for $t=1,\ldots,n$ the \emph{influence} of the variable $x_t$ is defined as
\[\on{Inf}_t[g] = \sum_{\substack{S\subseteq[n]\\S\ni t}} a_S^2.\]
\end{definition}
\begin{theorem}\label{thm:invariance}
Fix $M\ge 1$; there exists $M'>0$ such that the following holds. Let $g$ be an $n$-variable multilinear polynomial of degree at most $k$. Let $\vec{y}$ be uniformly random vector such that $\mb{E}[y_i] = 0$, $\mb{E}[y_i^2] = 1$ and $\mb{E}[|y_i|^3]\le M$. Let $\vec{z}\sim\mc{N}(0,1)^{\otimes n}$ be a vector of independent standard Gaussian random variables. Then for any three-times-differentiable function $\psi\colon\mb{R}\to\mb{R}$, we have \[\Big|\mb{E}[\psi(g(\vec{y}))-\psi(g(\vec{z}))]\Big|
\le (M')^k\cdot \snorm{\psi^{(3)}}_{\infty}\sum_{t=1}^n\on{Inf}_t[g]^{3/2}.\]
\end{theorem}

Next, we will require the following concentration inequality for low-degree polynomials of Gaussian; the Rademacher case is stated as \cite[Theorem~9.23]{O14} and the Gaussian case follows by taking limits via and applying the central limit theorem.
\begin{theorem}\label{thm:concentration-hypercontractivity}
Let $f$ be a polynomial in $n$ variables of degree at most $d$. Let $\vec{x}=(x_1,\ldots,x_n)$ either be a vector of independent standard Gaussian random variables. Then for any $t\ge (2e)^{d/2}$,
\[\Pr\left[|f(\vec x)|\ge t(\mb{E}[f(\vec{x})^2])^{1/2}\right]\le \exp\left(-\frac{d}{2e} t^{2/d}\right).\]
\end{theorem}

We also require a statement allowing one to quantify the convergence in distribution of a random variable given convergence of the associated Fourier transform. The following result is immediate from \cite[p.~104,~Theorem~1]{Pet12}; this is an essentially standard inequality used in the proof of the Berry--Esseen theorem.
\begin{theorem}\label{thm:fourier-convert}
There exists an absolute constant $C = C_{\ref{thm:fourier-convert}}>0$ such that the following statement holds. Consider a pair of random variables $X$ and $Y$ and a parameter $T>0$. We have that 
\[\sup_{\tau\in\mb{R}}|\mb{P}[X\le \tau]-\mb{P}[Y\le \tau]|\le C_{\ref{thm:fourier-convert}}\bigg(\int_{-T}^{T}\frac{|\mb{E}[\exp(itX)-\exp(itY)]|}{|t|}dt + \sup_{\tau\in\mb{R}}\mb{P}[|Y-\tau|\le 1/T]\bigg).\]
\end{theorem}

Next we will require a multidimensional version of Ess\'{e}en's concentration inequality.
\begin{theorem}[{\cite[Lemma~7.17]{TV10}}]\label{thm:esseen}
There exists an absolute constant $C = C_{\ref{thm:esseen}}>0$ such that the following statement holds. Given a random variable $X$ in $\mb{R}^d$, we have that 
\[\sup_{\tau\in\mb{R}^d}\mb{P}[\snorm{X-\tau}_2\le \eps] \le \bigg(\frac{C_{\ref{thm:esseen}}\eps}{\sqrt{d}}\bigg)^{d}\int_{\substack{\vec{\xi}\in \mb{R}^d\\\snorm{\vec{\xi}}_2\le d/\eps}}|\mb{E}[\exp(2\pi i \vec{\xi} \cdot X)]|d\vec{\xi}.\]
\end{theorem}

Finally we will require the following consequence of Fourier inversion on lattices.
\begin{theorem}\label{thm:inversion-formula}
Given a bounded random variables $T\in\mb{Z}^d$ and $X\in\mb{R}$, possibly dependent, we have 
\[\mb{E}[\mbm{1}_{T = \vec{t}}X] = (2\pi)^{-d}\int_{[-\pi,\pi]^{d}}\exp(-i\vec{t}\cdot\vec{\xi})\mb{E}[X\exp(i\vec{\xi}\cdot T)]d\vec{\xi}.\]
\end{theorem}

\section{Fourier coefficient bounds}\label{sec:fourier}
\subsection{Coefficient sequence}
For the purposes of proving various central limit theorem and local central limit theorems, we will consider sums
\[\sum_{j=1}^nc_j\wt{a}_j\]
with coefficient sequences $(c_j)_{1\le j\le n}$ which are more general than those arising from $\big(\sum_{1\le k<j}\wt{b}_j + \frac{\wt{b}_j}{2} - (j-1/2)\big)$, which comes out of \cref{lem:beats}. The following definitions for such sequences arises from the proof; roughly, a sequence is \emph{well-bounded} if it does not deviate much more than a simple random walk would, and it is \emph{coarse} if it further resembles such a simple random walk at some finer scales.
\begin{definition}\label{def:coarse}
We say a coefficient sequence $(c_j)_{1\le j\le n}$ is \emph{well-bounded} if the following conditions hold:
\begin{enumerate}[{\bfseries{S\arabic{enumi}}}]
    \item\label{S1} $|c_j|\le\sqrt{n}\log n$;
    \item\label{S2} $\sum_{j=1}^nc_j = 0$;
    \item\label{S3} $|c_j-c_k|\le\sqrt{|j-k|}(\log n)^2$ for all $1\le j,k\le n$;
\end{enumerate}
and we say it is \emph{coarse} if it is well-bounded and additionally the following hold:
\begin{enumerate}[{\bfseries{S\arabic{enumi}}}]
    \setcounter{enumi}{3}
    \item\label{S4} $\min_{a,b\in\mb{R}}\sum_{j=1}^n(c_j-aj-b)^2\ge n^2/(\log n)^2$;
    \item\label{S5} There are at least $n/\log n$ indices $j$ such that $c_j = c_{j+1} = c_{j+2} - 1/2$;
    \item\label{S6} For each integer $y\in [n^{1/4},n/(\log n)^2]$ there are at least $n/\log n$ indices $1\le j\le n-2y$ such that $|c_j-2c_{j+y}+c_{j+2y}|\ge\sqrt{y}$.
\end{enumerate}
\end{definition}

\subsection{Fourier estimates}\label{sub:fourier-estimates}
We now bound various Fourier expressions that will serve as a key input to our argument. We first define the basic setup.
\begin{definition}\label{def:triple-var-setup}
Let $\Delta$ be a distribution  which is either $\mr{Geom}(1/2)$ or $\mr{Pois}(1)$. Sample $X_j\sim\Delta$ independently for $1\le j\le n$ and fix a sequence $(c_j)_{1\le j\le n}$. Define the random variables
\[T_1 = \sum_{j=1}^nX_j - n,\qquad T_2 = \sum_{j=1}^njX_j-\frac{n(n+1)}{2},\qquad T_3 = 2\sum_{j=1}^nc_jX_j.\]
\end{definition}

We will be interested in Fourier coefficients of the form $\mb{E}\exp(i\vec{\Theta}\cdot(T_1,T_2,T_3))$. Our approach in general will be to reduce to essentially expressions involving Rademacher random variables and then to apply various basic bounds to conclude. (Note that we are not conditioning on the sum variable $T_1$ or ``area'' variable $T_2$ at this stage.)

\begin{fact}\label{fact:estimate}
Given $R\sim\mr{Ber}(1/2)$, $R\sim\mr{Geom}(1/2)$, or $R\sim\mr{Pois}(1)$ and $|\Theta|\le 3\pi/2$ we have 
\[|\mb{E}\exp(iR\Theta)|\le\exp(-c_{\ref{fact:estimate}}\Theta^2)\]
for some appropriate absolute constant $c_{\ref{fact:estimate}}>0$.
\end{fact}
\begin{proof}
This follows immediately from the explicit computation that
\begin{align*}
|\mb{E}\exp(i\Theta R)|=\begin{cases}|\cos(\Theta/2)|,&\text{if }R\sim\mr{Ber}(1/2),\\|2-\exp(i\Theta)|^{-1},&\text{if }R\sim\mr{Geom}(1/2),\\\exp(\cos(\Theta)-1),&\text{if }R\sim\mr{Pois}(1)\end{cases}
\end{align*}
and some simple bounds based on Taylor series.
\end{proof}

We first handle $\vec{\Theta}$ where $|\Theta_3|$ is large, since it is the most involved and serves as a basis for the other proofs. The key idea, which will be used to handle all the estimates present, is to extract independent random variables which isolate the effect of exactly one of the $\Theta_j$.

\begin{lemma}\label{lem:theta3-high}
Suppose that $\vec{\Theta} = (\Theta_1,\Theta_2,\Theta_3)$ is such that $n^{-1/2}(\log n)^2\le|\Theta_3|\le\pi$. Then given \cref{def:triple-var-setup} and that $(c_j)_{1\le j\le n}$ is coarse, we have
\[|\mb{E}\exp(i\vec{\Theta}\cdot(T_1,T_2,T_3))|\le n^{-\omega(1)}.\]
\end{lemma}
\begin{remark}
This estimate, as well as \cref{lem:theta3-mid}, is only needed to establish \cref{thm:ties}.
\end{remark}
\begin{proof}
Since the coefficient sequence $(c_j)_{1\le j\le n}$ is coarse, using \cref{S5} there exists a $4$-separated set of indices $J$ (i.e., the difference of distinct elements is at least $4$) such that $|J|=\Omega(n/\log n)$ and such that for $j\in J$ we have $c_j = c_{j+1} = c_{j+2}-1/2$. We now claim that
\begin{equation}\label{eq:rerandomize}
(X_1,X_2,X_3) \overset{d.}{=} (1-W)Z + W((1-R)(0,2,0) + R(1,0,1))
\end{equation}
where $R$, $W$, and $Z$ are independent random variables defined via $R=\mr{Ber}(1/2)$,
\begin{align*}
W &= \mr{Ber}\big(\min\{\mb{P}[(X_1,X_2,X_3) = (0,2,0)], \mb{P}[(X_1,X_2,X_3) = (1,0,1)]\}\big),\\
\mb{P}[Z = (k_1,k_2,k_3)] &= \frac{1}{\mb{E}[1-W]}\cdot\bigg(\mb{P}[(X_1,X_2,X_3) = (k_1,k_2,k_3)] - \mbm{1}_{(k_1,k_2,k_3)\in \{(0,2,0),(1,0,1)\}}\mb{E}W\bigg)
\end{align*}
for $(k_1,k_2,k_3)\in\mb{Z}_{\ge 0}^3$. Indeed, to see this let $2q=\min\{\mb{P}[(X_1,X_2,X_3) = (0,2,0)], \mb{P}[(X_1,X_2,X_3) = (1,0,1)]\}$ and consider the following procedure: sample $(X_1,X_2,X_3)$, but if either of the tuples $(x_1,x_2,x_3)\in\{(0,2,0),(1,0,1)\}$ is drawn then with probability $q/\mb{P}[(X_1,X_2,X_3)=(x_1,x_2,x_3)]$ enter a ``resampling phase'' where we with probability $1/2$ decide whether to output $(0,2,0)$ or $(1,0,1)$, overwriting the old value to produce a tuple $(X_1',X_2',X_3')$. (So, the ``resampling phase'' occurs with chance $2q$ by the law of total probability.) We see the distributional equality $(X_1',X_2',X_3')\overset{d.}=(X_1,X_2,X_3)$ by construction, but $(X_1',X_2',X_3')$ is easily seen to be captured by the formula \cref{eq:rerandomize}.

Note \cref{eq:rerandomize} holds even if we shift indices, so for each $j\in J$ we can write $(X_j,X_{j+1},X_{j+2}) = (1-W_j)Z_j + W_j((0,2,0) + R_j(1,-2, 1))$. Notice by the triangle inequality and independence that 
\begin{align*}
|\mb{E}\exp(i\vec{\Theta}\cdot (T_1,T_2,T_3))| &\le \Big|\mb{E}\Big[\exp\Big(i\sum_{j\in J+\{0,1,2\}}\Big(\Theta_1 + \Theta_2j + 2\Theta_3 c_j\Big)X_j\Big)\Big]\Big|\\
&\le\mb{E}_W\Big[\Big|\mb{E}\Big[\exp\Big(i\sum_{j\in J+\{0,1,2\}}\Big(\Theta_1 + \Theta_2j + 2\Theta_3 c_j\Big)X_j\Big)\Big|(W_j)_{j\in J}\Big]\Big|\Big]\\
&= \mb{E}_W\Big[\Big|\mb{E}\Big[\exp\Big(i\sum_{j\in J}\mbm{1}_{W_j = 1}\Theta_3R_j\Big)\Big|(W_j)_{j\in J}\Big]\Big|\Big]\\
&\le \mb{E}[\exp(-\Omega(\Theta_3^2)\sum_{j\in J}\mbm{1}_{W_j = 1})]\\
&\le n^{-\omega(1)}.
\end{align*}
The first and second lines follow from independence and the triangle inequality, the third follows from
\[(1,1,1)\cdot(1,-2,1) = 0,\quad (j,j+1,j+2)\cdot (1,-2,1)) = 0,\quad (c_j,c_{j+1},c_{j+2})\cdot(1,-2,1) = -1/2,\]
and the fourth follows from \cref{fact:estimate}. In the final line we have used independence and Bernstein's inequality, which implies that $\#\{j\in J\colon W_j=1\}=\sum_{j\in J}\mbm{1}_{W_j=1}\ge cn/\log n$ occurs with super-polynomially small probability for some small absolute constant $c>0$.
\end{proof}

We next handle the case of intermediate $|\Theta_3|$. The remaining unhandled range will be in some sense controllable by an appropriate central limit theorem.
\begin{lemma}\label{lem:theta3-mid}
Suppose that $\vec{\Theta} = (\Theta_1,\Theta_2,\Theta_3)$ is such that $n^{-1}(\log n)^3\le|\Theta_3|\le n^{-1/2}(\log n)^2$. Then given \cref{def:triple-var-setup} and that $(c_j)_{1\le j\le n}$ is coarse, we have
\[|\mb{E}\exp(i\vec{\Theta}\cdot(T_1,T_2,T_3))|\le n^{-\omega(1)}.\]
\end{lemma}
\begin{proof}
Let $y\in[n^{1/4},n/(\log n)^2]$ be an integer to be chosen later based on $n,|\Theta_3|$. Since $(c_j)_{1\le j\le n}$ is a coarse sequence, by \cref{S6} there exists a set of indices $J$ of size $\Omega(n/\log n)$ such that the sets $J, J+y, J+2y$ are disjoint and for each $j\in J$ we have $\sqrt{y}\le|c_j-2c_{j+y}+c_{j+2y}|\le 2\sqrt{y}(\log n)^2$ (the second inequality follows from two applications of \cref{S3}). Therefore proceeding in an essentially identical manner to \cref{lem:theta3-high} (in particular writing $(X_j,X_{j+y},X_{j+2y})=(1-W_j)Z_j+W_j((0,2,0)+R_j(1,-2,1))$ for $j\in J$ similar to the proof of the previous lemma), we have that
\begin{align*}
|\mb{E}[\exp(i\vec{\Theta}\cdot (T_1,T_2,T_3))]| &\le \Big|\mb{E}\Big[\exp\Big(i\sum_{j\in J+\{0,y,2y\}}\Big(\Theta_1 + \Theta_2j + 2\Theta_3 c_j\Big)X_j\Big)\Big]\Big|\\
&\le \mb{E}\Big[\Big|\mb{E}\Big[\exp\Big(i\sum_{j\in J+\{0,1,2\}}\Big(\Theta_1 + \Theta_2j + 2\Theta_3 c_j\Big)X_j\Big)\Big|(W_j)_{j\in J}\Big]\Big|\Big]\\
&= \mb{E}\Big[\Big|\mb{E}\Big[\exp\Big(i\sum_{j\in J}\mbm{1}_{W_j = 1}\Theta_3R_j(c_j-2c_{j+y}+c_{j+2y})\Big)\Big|(W_j)_{j\in J}\Big]\Big|\Big]\\
&\le \mb{E}[\exp(-\Omega(\Theta_3^2y)\sum_{j\in J}\mbm{1}_{W_j = 1})]\\
&\le n^{-\omega(1)}.
\end{align*}
The reasoning is essentially identical to that in the proof of \cref{lem:theta3-high}. We need that $|\Theta_3(c_j-2c_{j+y}+c_{j+2y})|\le \sqrt{y}(\log n)^2|\Theta_3|\le 1$ in order to apply \cref{fact:estimate}. If we additionally have that $\Theta_3^2y\cdot n/(\log n)\ge (\log n)^2$, then using $|J|=\Omega(n/\log n)$ we can conclude the final estimate in a similar manner to the proof of \cref{lem:theta3-high}. Thus it suffices to choose an integer $y$ satisfying
\[(\log n)^3/(n\Theta_3^2)\le y\le 1/((\log n)^2\Theta_3^2)\]
and $y\in[n^{1/4},n/(\log n)^2]$. This clearly exists by the given bounds for $|\Theta_3|$.
\end{proof}

We now prove a similar estimate for the case where $|\Theta_2|$ is near the maximum size. The proof is once again rather similar, but in this case we only need to consider consecutive pairs of indices $(j,j+1)$ in order to extract the necessary effect.
\begin{lemma}\label{lem:theta2-high}
Suppose that $\vec{\Theta} = (\Theta_1,\Theta_2,\Theta_3)$ is such that $n^{-1/2}\log n\le |\Theta_2|\le\pi$, $|\Theta_3|\le n^{-1}(\log n)^3$ and $(c_j)_{1\le j\le n}$ satisfies \cref{S1,S3}. Then given \cref{def:triple-var-setup}, we have
\[|\mb{E}\exp(i\vec{\Theta}\cdot (T_1,T_2,T_3))|\le n^{-\omega(1)}.\]
\end{lemma}
\begin{remark}
\cref{lem:theta2-high,lem:theta2-mid,lem:theta1-region} are needed for both \cref{thm:main,thm:ties}. Note that these lemmas do not need an assumption on coarseness of $(c_j)_{1\le j\le n}$.
\end{remark}
\begin{proof}
Since the coefficient sequence $(c_j)_{1\le j\le n}$ satisfies \cref{S3} we have $|c_j-c_{j+1}|\le(\log n)^2$. Let $J\subseteq[n]$ be a $2$-separated set of indices of size $\Omega(n)$. Furthermore note that $(X_1,X_2) \overset{d.}{=} (1-W)Z + W((1,0) + R(-1,1))$ where $R,W,Z$ are independent random variables with $R=\mr{Ber}(1/2)$, $W = \mr{Ber}(\mb{P}[(X_1,X_2)=(1,0)])$, and
\[\mb{P}[Z = (k_1,k_2)] = \frac{1}{\mb{E}[1-W]}\cdot\bigg(\mb{P}[(X_1,X_2,X_3) = (k_1,k_2)] - \mbm{1}_{(k_1,k_2)\in\{(0,1),(1,0)\}}\mb{E}W\bigg)\]
for $(k_1,k_2)\in\mb{Z}_{\ge 0}^2$, similar to as in the proof of \cref{lem:theta3-high}.

Now for each index in $J$, we write $(X_j,X_{j+1}) = (1-W_j)Z_j + W_j((1,0) + R_j(1,-1))$. Notice by the triangle inequality and independence that 
\begin{align*}
|\mb{E}\exp(i\vec{\Theta}\cdot (T_1,T_2,T_3))| &\le \Big|\mb{E}\Big[\exp\Big(i\sum_{j\in J+\{0,1\}}\Big(\Theta_1 + \Theta_2j + 2\Theta_3 c_j\Big)X_j\Big)\Big]\Big|\\
&\le \mb{E}\Big[\Big|\mb{E}\Big[\exp\Big(i\sum_{j\in J+\{0,1\}}\Big(\Theta_1 + \Theta_2j + 2\Theta_3 c_j\Big)X_j\Big)\Big|(W_j)_{j\in J}\Big]\Big|\Big]\\
&= \mb{E}\Big[\Big|\mb{E}\Big[\exp\Big(i\sum_{j\in J}\mbm{1}_{W_j = 1}(\Theta_2+2(c_{j+1}-c_j)\Theta_3)R_j\Big)\Big|(W_j)_{j\in J}\Big]\Big|\Big]\\
&\le \mb{E}[\exp(-\Omega(\Theta_2^2)\sum_{j\in J}\mbm{1}_{W_j = 1})]\\
&\le n^{-\omega(1)}.
\end{align*}
The first and second line follows from triangle inequality, the third follows from $(1,1)\cdot(1,-1) = 0$, $(j,j+1)\cdot (1,-1)) = -1$, and $(c_j,c_{j+1})\cdot(1,-1) = c_{j+1}-c_j$ for $j\in J$, and the fourth from \cref{fact:estimate} as well as $2|c_{j+1}-c_j||\Theta_3|\le 2n^{-1}(\log n)^5\le |\Theta_2|/2$. In the final line we have once again used Bernstein's inequality.
\end{proof}

We next handle intermediate $|\Theta_2|$. In the remaining range central limit theorem type estimates become effective.
\begin{lemma}\label{lem:theta2-mid}
Suppose that $\vec{\Theta} = (\Theta_1,\Theta_2,\Theta_3)$ is such that $n^{-3/2}(\log n)^6\le |\Theta_2|\le n^{-1/2}\log n$, $|\Theta_3|\le n^{-1}(\log n)^3$ and $(c_j)_{1\le j\le n}$ satisfies \cref{S1,S3}. Then given \cref{def:triple-var-setup}, we have 
\[|\mb{E}\exp(i\vec{\Theta}\cdot (T_1,T_2,T_3))|\le n^{-\omega(1)}.\]
\end{lemma}
\begin{proof}
Let $1\le y\le n/8$ be an integer to be chosen later based on $n,|\Theta_2|$. Consider $J = \{\lfloor n/2\rfloor-2y,\lfloor n/2\rfloor-2y + 1,\ldots,\lfloor n/2\rfloor-y\}$. We have
\[|\Theta_3||c_j-c_{n-j}|\lesssim\sqrt{y}(\log n)^5/n,\quad|\Theta_2||n-2j| \asymp y|\Theta_2|\]
for all $j\in J$. We ensure that $y$ is chosen so that $\sqrt{y}(\log n)^5/n\le cy|\Theta_2|$ for an appropriately small absolute constant $c>0$ and so that $y|\Theta_2|\le c$ as well. We also guarantee $y\ge(\log n)^2$.

We now write $(X_j,X_{n-j}) = (1-W_j)Z_j + W_j((1,0) + R_j(-1,1))$ for $j\in J$ in a similar manner to the proof of \cref{lem:theta2-high}, and find
\begin{align*}
|\mb{E}\exp(i\vec{\Theta}\cdot (T_1,T_2,T_3))| &\le \Big|\mb{E}\Big[\exp\Big(i\sum_{j\in J\cup (n-J)}\Big(\Theta_1 + \Theta_2j + 2\Theta_3 c_j\Big)X_j\Big)\Big]\Big|\\
&\le \mb{E}\Big[\Big|\mb{E}\Big[\exp\Big(i\sum_{j\in J\cup (n-J)}\Big(\Theta_1 + \Theta_2j + 2\Theta_3 c_j\Big)X_j\Big)\Big|(W_j)_{j\in J}\Big]\Big|\Big]\\
&= \mb{E}\Big[\Big|\mb{E}\Big[\exp\Big(i\sum_{j\in J}\mbm{1}_{W_j = 1}((n-2j)\Theta_2 + 2(c_{n-j}-c_j)\Theta_3)R_j\Big)\Big|(W_j)_{j\in J}\Big]\Big|\Big]\\
&\le \mb{E}[\exp(-\Omega(y^2\Theta_2^2)\sum_{j\in J}\mbm{1}_{W_j = 1})]\\
&\le n^{-\omega(1)}.
\end{align*}
We used that $(n-2j)\Theta_2$ dominates $2(c_{n-j}-c_j)\Theta_3$ in the second last line, as well as $2|(n-2j)\Theta_2|\le 1$ to apply \cref{fact:estimate}. For the last line, we note that $\sum_{j\in J}\mbm{1}_{W_j = 1}\gtrsim y$ occurs with super-polynomially good probability (since $y\ge(\log n)^2$) and we are using the estimate $y^3\Theta_2^2\ge (\log n)^2$.

To finish the proof, we check that it is possible to choose integer $1\le y\le n/8$ with $y\ge(\log n)^{2/3}\Theta_2^{-2/3}$ and $y\ge(\log n)^2$ as well as $y\le c|\Theta_2|^{-1}$ and $y\ge c^{-2}(\log n)^{10}/(n^2\Theta_2^2)$. The bounds on $|\Theta_2|$ easily imply this is possible.
\end{proof}

We now are finally in position to handle the cases where $|\Theta_1|$ is large. The remaining region will be handled by central limit theorem style techniques.
\begin{lemma}\label{lem:theta1-region}
Suppose that $\vec{\Theta} = (\Theta_1,\Theta_2,\Theta_3)$ is such that $n^{-1/2}(\log n)^7\le |\Theta_1|\le5\pi/4$, $|\Theta_2|\le n^{-3/2}(\log n)^6$, and $|\Theta_3|\le n^{-1}(\log n)^3$ and $(c_j)_{1\le j\le n}$ satisfies \cref{S1,S3}. Then given \cref{def:triple-var-setup}, we have
\[|\mb{E}\exp(i\vec{\Theta}\cdot (T_1,T_2,T_3))|\le n^{-\omega(1)}.\]
\end{lemma}
\begin{proof}
Note that $|\Theta_2|n + |\Theta_3|\max_{1\le j\le n}|c_j|\le2(\log n)^6n^{-1/2}$. Therefore we have
\begin{align*}
|\mb{E}[\exp(\vec{\Theta}\cdot (T_1,T_2,T_3))]| &=\Big|\mb{E}\Big[\exp\Big(i\sum_{j=1}^n\Big(\Theta_1 + \Theta_2j + 2\Theta_3 c_j\Big)X_j\Big)\Big]\Big|\\
&\le \mb{E}[\exp(-\Omega(n\Theta_3^2))]\le n^{-\omega(1)}.
\end{align*}
where we have simply noted that $\Theta_1$ dominates $\Theta_2j+2\Theta_3c_j$ and applied \cref{fact:estimate}.
\end{proof}

We now prove the desired estimate for the region which is approximately within the region which is controlled via a central limit theorem. For completeness we provide a short proof via an argument closely related to the Lindeberg exchange method \cite{Lin22} and the proof of the Gaussian invariance principle \cite{MOO10} (see \cref{thm:invariance}). This will help us reduce computing the necessary integrals to a purely Gaussian integration problem.
\begin{lemma}\label{lem:clt-region}
Suppose that $\vec{\Theta} = (\Theta_1,\Theta_2,\Theta_3)$ is such that $|\Theta_1|\le n^{-1/2}(\log n)^7$, $|\Theta_2|\le n^{-3/2}(\log n)^6$,  $|\Theta_3|\le n^{-1}(\log n)^3$, and $(c_j)_{1\le j\le n}$ satisfies \cref{S1,S2}. Given \cref{def:triple-var-setup}, we further define 
\[T_1' = \sum_{j=1}^nX_j',\qquad T_2' = \sum_{j=1}^njX_j',\qquad T_3' = 2\sum_{j=1}^nc_jX_j'\]
where we independently sample $X_j'\sim\mc{N}(0,\mr{Var}[\Delta])$. Then we have 
\[|\mb{E}[\exp(i\vec{\Theta}\cdot (T_1,T_2,T_3))] - \mb{E}[\exp(i\vec{\Theta}\cdot (T_1',T_2',T_3'))]|\lesssim n^{-1/2}(\log n)^{21}.\]
\end{lemma}
\begin{remark}
If $\Delta=\mr{Geom}(1/2)$ then $(\mb{E}[\Delta],\mr{Var}[\Delta])=(1,2)$, and if $\Delta=\mr{Pois}(1)$ then $(\mb{E}[\Delta],\mr{Var}[\Delta])=(1,1)$.
\end{remark}
\begin{proof}
Notice that by iteratively replacing $X_i$ by $X_i'$ and applying the triangle inequality we have
\begin{align*}
&|\mb{E}\exp(i\vec{\Theta}\cdot (T_1,T_2,T_3))-\mb{E}\exp(i\vec{\Theta}\cdot (T_1',T_2',T_3'))|\\
&\le\sum_{j=1}^n\bigg|\bigg(\mb{E}\exp(i\vec{\Theta}\cdot (1,j,2c_j)(X_j-1))-\mb{E}\exp(i\vec{\Theta}\cdot (1,j,2c_j)X_j')\bigg)\\
&\qquad\qquad\times\mb{E}\exp\bigg(i\vec{\Theta}\cdot\bigg(\sum_{1\le j'<j}(1,j',2c_{j'})X_{j'}'+\sum_{j<j'\le n}(1,j',2c_{j'})(X_{j'}-1)\bigg)\bigg)\bigg|\\
&\le\sum_{j=1}^n|\mb{E}\exp(i\vec{\Theta}\cdot (1,j,2c_j)(X_j-1))-\mb{E}\exp(i\vec{\Theta}\cdot (1,j,2c_j)X_j')|\\
&\lesssim\sum_{j=1}^n(|\Theta_1| + n|\Theta_2| + n^{1/2}\log n |\Theta_3|)^3\mb{E}[|X_j|^3 + |X_j'|^3]\lesssim n^{-1/2}(\log n)^{21}.
\end{align*}
To justify the second-to-last inequality, we use that $|\exp(ix) - 1 - ix + x^2/2|\le |x|^3$ for $x\in\mb{R}$ from Taylor's theorem and that the first and second moments of $X_j-1$ and $X_j'$ match.
\end{proof}

\section{Translating Fourier information}\label{sec:translate-fourier}
We now translate Fourier information into probabilistic information in order to prove \cref{thm:main}. We defer the proof of the following lemma, which shows that certain coefficient sequences that will arise in our computation are well-bounded with very good probability, until the next section.
\begin{lemma}\label{lem:coefficient-combinations}
Fix $m$ and let $1\le k^\ast\le m$. Consider $\Theta\neq 0$ and $\theta=(\theta_{jk})_{1\le j<k\le m}$ with $\snorm{\theta}_\infty\le\Theta$. Consider independent random variables $X_j^{(k)}\sim\Delta$ for $1\le k\le m$ and $1\le j\le n$, where $\Delta\in\{\mr{Geom}(1/2),\mr{Pois}(1)\}$. Finally, let
\[c_j^{(k^\ast)}=\frac{1}{2\Theta}\bigg(\sum_{k<k^\ast}\theta_{kk^\ast}\bigg(\sum_{j'=1}^n(M_n^\ast)_{jj'}(X_{j'}^{(k)}-1)\bigg)+\sum_{k>k^\ast}\theta_{k^\ast k}\bigg(\sum_{j'=1}^n(M_n^\ast)_{j'j}(X_{j'}^{(k)}-1)\bigg)\bigg).\]
Then for each with probability $1-n^{-\omega(1)}$ we have that $(c_j^{(k^\ast)})_{1\le j\le n}$ is well-bounded (\cref{def:coarse}).
\end{lemma}

Now we prove \cref{thm:main}.
\begin{proof}[Proof of \cref{thm:main} given \cref{lem:coefficient-combinations}]
Sample $k$ dice either all from the multiset model or all from the balanced sequence model, $A_1,\ldots,A_m$ with $A_j=(a_{j1},\ldots,a_{jn})\in[n]^n$. Let the frequency counts of die $A_k$ be $\wt{a}_{ki}=|\{j\colon a_{kj}=i\}|$. We are given the tournament $D$ on $[k]$ and wish to understand the chance that $A_i$ beats $A_j$ precisely when $ij$ is a directed edge. (Note that we may assume $D$ is a full tournament since partial tournaments clearly follow by summing appropriately.)

If we are in the multiset model let $\Delta=\mr{Geom}(1/2)$ and if we are in the balanced sequence model let $\Delta=\mr{Pois}(1)$. Consider $m$ independent copies of the setup in \cref{def:triple-var-setup} (ignoring the sequence $c_j$ and random variable $T_3$), denoted by $(X_j^{(k)})_{1\le j\le n}$ and $(T_j^{(k)})_{1\le j\le 2}$ for $1\le k\le m$, corresponding to the $k$ dice. Note that $(\wt{a}_{ki})_{1\le i\le n}$ is distributed as $(X_j^{(k)})_{1\le j\le n}$ conditional on $T_1^{(k)}=T_2^{(k)}=0$ by \cref{lem:crucial}.

Let
\[Y_{k_1k_2}:=X^{(k_2)T}M_n^\ast X^{(k_1)}=(X^{(k_2)}-\sqrt{n}\vec{v}_1)^TM_n^\ast(X^{(k_1)}-\sqrt{n}\vec{v}_1)\]
for $1\le k_1<k_2\le m$ (recall $\vec{v}_1$ from \cref{def:dice-matrix}), and for $\theta=(\theta_{k_1k_2})_{1\le k_1<k_2\le m}$ let
\[Y(\theta):=\exp\bigg(i\sum_{1\le k_1<k_2\le m}\theta_{k_1k_2}Y_{k_1k_2}\bigg).\]
By \cref{thm:inversion-formula} with $T:=(T_b^{(k)})_{k\in[m],b\in[2]}$ and indexing the coordinates of $\vec{\xi}$ by $(\xi_{kb})_{k\in[m],b\in[2]}$, we have
\begin{align}
\mb{E}[\mbm{1}_{T=\vec{0}}Y(\theta)]&=(2\pi)^{-2m}\int_{[-\pi,\pi]^{2m}}\mb{E}[Y(\theta)\exp(i\vec{\xi}\cdot T)]d\vec{\xi}\notag\\
&=(2\pi)^{-2m}\int_{[-\pi,\pi]^{2m}}\mb{E}\bigg[\exp\bigg(i\sum_{1\le k_1<k_2\le m}\theta_{k_1k_2}Y_{k_1k_2}+i\vec{\xi}\cdot T\bigg)\bigg]d\vec{\xi}.\label{eq:main-fourier-inversion}
\end{align}

We fix some $\theta$ satisfying $\snorm{\theta}_\infty\le n^{-1}(\log n)^3$. Given this condition, we will now estimate the integrand and show that it is very small unless $\snorm{\vec{\xi}_{\cdot1}}_\infty=\wt{O}(n^{-1/2})$ and $\snorm{\vec{\xi}_{\cdot2}}_\infty=\wt{O}(n^{-3/2})$.

We now collect terms so as to express the argument in the exponential as a linear function of $X^{(k^\ast)}$ with coefficients depending on $X^{(k)}$ for $k\neq k^\ast$. We see
\begin{align}
&\sum_{1\le k_1<k_2\le m}\theta_{k_1k_2}Y_{k_1k_2}+\vec{\xi}\cdot T\notag\\
&=\sum_{j=1}^n\bigg(\sum_{k<k^\ast}\theta_{kk^\ast}\bigg(\sum_{j'=1}^n(M_n^\ast)_{jj'}(X_{j'}^{(k)}-1)\bigg)+\sum_{k>k^\ast}\theta_{k^\ast k}\bigg(\sum_{j'=1}^n(M_n^\ast)_{j'j}(X_{j'}^{(k)}-1)\bigg)\bigg)X_j^{(k^\ast)}+\wt{Y}_{k^\ast}\label{eq:fourier-linearization}
\end{align}
for some $\wt{Y}_{k^\ast}$ that depends only on $(X^{(k)})_{k\neq k^\ast}$. Consider $\Theta=(\log n)^3/n$, and define
\[c_j^{(k^\ast)}:=\frac{1}{2\Theta}\bigg(\sum_{k<k^\ast}\theta_{kk^\ast}\bigg(\sum_{j'=1}^n(M_n^\ast)_{jj'}(X_{j'}^{(k)}-1)\bigg)+\sum_{k>k^\ast}\theta_{k^\ast k}\bigg(\sum_{j'=1}^n(M_n^\ast)_{j'j}(X_{j'}^{(k)}-1)\bigg)\bigg)\]
for $k^\ast\in[m]$ and $j\in[n]$. We have
\begin{align}
\bigg|\mb{E}\exp\bigg(i\sum_{1\le k_1<k_2\le m}\theta_{k_1k_2}Y_{k_1k_2}+i\vec{\xi}\cdot T\bigg)\bigg|&=\bigg|\mb{E}\exp\bigg(i\bigg(\Theta\sum_{j=1}^n2c_j^{(k^\ast)}X_j^{(k^\ast)}+\xi_{k^\ast,1}\sum_{j=1}^n(X_j^{(k^\ast)}-1)\notag\\
&\qquad\qquad\qquad+\xi_{k^\ast,2}\sum_{j=1}^nj(X_j^{(k^\ast)}-1)+\wt{Y}_{k^\ast}\bigg)\bigg)\bigg|.\label{eq:fourier-combined}
\end{align}

Now we apply \cref{lem:theta2-high,lem:theta2-mid,lem:theta1-region} to gain control over $\xi$. In order to use these, we need each $(c_j^{(k^\ast)})_{1\le j\le n}$ for $k^\ast\in[m]$ to be a well-bounded coefficient sequence. By \cref{lem:coefficient-combinations}, this occurs with super-polynomially good probability over $(X^{(k)})_{k\neq k^\ast}$.

So if $n^{-1/2}\log n\le|\xi_{k^\ast 2}|\le\pi$ then by \cref{lem:theta2-high} we have that \cref{eq:fourier-combined} is of magnitude $n^{-\omega(1)}$: condition on an outcome of $(X^{(k)})_{k\neq k^\ast}$ for which $c_j^{(k^\ast)}$ is well-bounded using \cref{lem:coefficient-combinations}, and then apply \cref{lem:theta2-high}. We are using that $\Theta=n^{-1}(\log n)^3$. Similarly, if $n^{-3/2}(\log n)^6\le|\xi_{k^\ast 2}|\le n^{-1/2}\log n$ then by \cref{lem:coefficient-combinations,lem:theta2-mid} we see that \cref{eq:fourier-combined} is of magnitude $n^{-\omega(1)}$. Finally, if $|\xi_{k^\ast 2}|\le n^{-3/2}(\log n)^6$ and $n^{-1/2}(\log n)^7\le|\xi_{k^\ast 1}|\le\pi$ then \cref{lem:coefficient-combinations,lem:theta1-region} show \cref{eq:fourier-combined} is of magnitude $n^{-\omega(1)}$.

Combining these observations with \cref{eq:main-fourier-inversion,eq:fourier-linearization} we see
\begin{equation}\label{eq:truncated-fourier}
\mb{E}[\mbm{1}_{T=\vec{0}}Y(\theta)]=(2\pi)^{-2m}\int_{[-\tau_1,\tau_1]^m\times[-\tau_2,\tau_2]^m}\mb{E}\bigg[\exp\bigg(i\sum_{1\le k_1<k_2\le m}\theta_{k_1k_2}Y_{k_1k_2}+i\vec{\xi}\cdot T\bigg)\bigg]d\vec{\xi}\pm n^{-\omega(1)},
\end{equation}
where $\tau_1=n^{-1/2}(\log n)^7$ and $\tau_2=n^{-3/2}(\log n)^6$. Additionally, the product in the region of integration is interpreted as corresponding to the choice of $b\in\{1,2\}$, i.e., the region is defined by $|\xi_{k1}|\le\tau_1$ and $|\xi_{k2}|\le\tau_2$.

Recall also that we assumed $\snorm{\theta}_\infty\le n^{-1}(\log n)^3$. We can now use an approach similar to the proof of \cref{lem:clt-region} (or \cite{Lin22,MOO10}) to exchange the variables $X_j^{(k)}$ with shifted Gaussians $Z_j^{(k)}+1$ where $Z_j^{(k)}\sim\mc{N}(0,\mr{Var}[\Delta])$. Note that
\begin{align}
\sum_{1\le k_1<k_2\le m}\theta_{k_1k_2}Y_{k_1k_2}+\vec{\xi}\cdot T&=\sum_{1\le k_1<k_2\le m}\theta_{k_1k_2}(X^{(k_2)}-\sqrt{n}\vec{v}_1)^TM_n^\ast(X^{(k_1)}-\sqrt{n}\vec{v}_1)\notag\\
&\qquad\qquad+\sum_{k=1}^m\xi_{k1}\sum_{j=1}^n(X_j^{(k)}-1)+\sum_{k=1}^m\xi_{k2}\sum_{j=1}^nj(X_j^{(k)}-1).\label{eq:invariance-form}
\end{align}
Now since $X_j^{(k)}-1$ are independent and mean $0$, and have variance $\mr{Var}[\Delta]$, we are in position to apply \cref{thm:invariance}. We first compute that the influences for the degree $2$ multilinear polynomial corresponding to \cref{eq:invariance-form} are bounded by
\[O((\snorm{M_n^\ast}_{1\to2}^2+\snorm{M_n^{\ast T}}_{1\to2}^2)\cdot\snorm{\theta}_\infty^2+\snorm{\xi_{\cdot 1}}_\infty^2+n^2\snorm{\xi_{\cdot2}}_\infty^2)=O(n^{-1}(\log n)^{14})\]
using \cref{lem:operator-decay} (specifically, \cref{M4}).

Let $Z_j^{(k)}\sim\mc{N}(0,\mr{Var}[\Delta])$ be independent Gaussians and let
\[\wt{Z}(\vec{\xi}):=\sum_{1\le k_1<k_2\le m}\theta_{k_1k_2}Z^{(k_2)T}M_n^\ast Z^{(k_1)}+\sum_{k=1}^m\xi_{k1}\sum_{j=1}^nZ_j^{(k)}+\sum_{k=1}^m\xi_{k2}\sum_{j=1}^njZ_j^{(k)}.\]
By \cref{thm:invariance} and \cref{eq:truncated-fourier} we have
\begin{equation}\label{eq:gaussian-truncated}
\mb{E}[\mbm{1}_{T=\vec{0}}Y(\theta)]=(2\pi)^{-2m}\int_{[-\tau_1,\tau_1]^m\times[-\tau_2,\tau_2]^m}\mb{E}[\exp(i\wt{Z}(\vec{\xi}))]d\vec{\xi}\pm O((\tau_1\tau_2)^mn^{-1/2}(\log n)^{21}).
\end{equation}
Note that the latter two sums in $\wt{Z}$, which involve $\vec{\xi}$, only depend on $Z^{(k)}\cdot\vec{v}_1$ and $Z^{(k)}\cdot\vec{v}_2$ whereas the bilinear forms only depend on the projection of $Z^{(k)}$ to the orthogonal complement of $\mr{span}_{\mb{R}}\{\vec{v}_1,\vec{v}_2\}$ (by \cref{def:dice-matrix}). Therefore we see that the first sum is independent from the latter two. This means that the integrand in \cref{eq:gaussian-truncated} is the product of some constant and some multivariate Gaussian characteristic function.

Now, if $\snorm{\xi_{\cdot1}}_\infty\ge\tau_1$ or $\snorm{\xi_{\cdot2}}_\infty\ge\tau_2$, then easily we find there is some $k^\ast\in[m]$ with
\[\sum_{j=1}^n(\xi_{k1}+j\xi_{k2})^2\gtrsim(\log n)^{12}.\]
We therefore deduce that for such $\vec{\xi}$,
\[|\mb{E}[\exp(i\wt{Z}(\vec{\xi}))]|=\bigg|\mb{E}\exp\bigg(i\sum_{1\le k_1<k_2\le m}\theta_{k_1k_2}Z^{(k_2)T}M_n^\ast Z^{(k_1)}\bigg)\bigg|\cdot\exp(-\Omega((\log n)^{12}))\le n^{-\omega(1)}.\]
Furthermore, since the integrand is proportional to the characteristic function of some multivariate Gaussian, it is easy to see that the integral to infinity over such $\vec{\xi}$ is still $n^{-\omega(1)}$ in size. So, from \cref{eq:gaussian-truncated} we deduce
\begin{align*}
&\mb{E}[\mbm{1}_{T=\vec{0}}Y(\theta)]=(2\pi)^{-2m}\int_{\mb{R}^{2k}}\mb{E}[\exp(i\wt{Z}(\vec{\xi}))]d\vec{\xi}\pm O((\tau_1\tau_2)^mn^{-1/2}(\log n)^{21})\\
&=\mb{E}\exp\bigg(i\sum_{1\le k_1<k_2\le m}\theta_{k_1k_2}Z^{(k_2)T}M_n^\ast Z^{(k_1)}\bigg)\\
&\qquad\cdot(2\pi)^{-2m}\int_{\mb{R}^{2k}}\mb{E}\exp\bigg(i\bigg(\sum_{k=1}^m\xi_{k1}\sum_{j=1}^nZ_j^{(k)}+\sum_{k=1}^m\xi_{k2}\sum_{j=1}^njZ_j^{(k)}\bigg)\bigg)d\vec{\xi}\pm O((\tau_1\tau_2)^mn^{-1/2}(\log n)^{21}).
\end{align*}
Plugging in $\theta=\vec{0}$ and dividing, and noting that the integral in the last line is order $\Theta((n^{-1/2}\cdot n^{-3/2})^m)$ (treating $m$ as fixed), we deduce
\begin{equation}\label{eq:final-levy}
\mb{E}[Y(\theta)|T=\vec{0}]=\mb{E}\exp\bigg(i\sum_{1\le k_1<k_2\le m}\theta_{k_1k_2}Z^{(k_2)T}M_n^\ast Z^{(k_1)}\bigg)\pm O(n^{-1/2}(\log n)^{21+13m})
\end{equation}
for $\snorm{\theta}_\infty\le n^{-1}(\log n)^3$.

We wish to show
\[\bigg(\frac{Y_{jk}}{n\mr{Var}[\Delta]}\bigg)_{1\le j<k\le m}\overset{d.}{\rightarrow}(H_{jk})_{1\le j<k\le m}\]
since \cref{lem:beats-matrix} (and the facts $\vec{v}_1^TM_n^\ast=0$ and $M_n^\ast\vec{v}_1=0$) shows $A_j$ beats $A_k$ precisely when $Y_{jk}>0$. Now let $G^{(j)}$ and $H_{jk}$ be as in \cref{thm:main}. From \cref{eq:final-levy} and L\'evy continuity, we see it is enough to show
\[\bigg(\frac{Z^{(k)T}M_n^\ast Z^{(j)}}{n\mr{Var}[\Delta]}\bigg)_{1\le j<k\le m}\overset{d.}{\rightarrow}(H_{jk})_{1\le j<k\le m}\]
as $n\to\infty$. (Simple inspection of the proof shows that this would also imply the second remark following \cref{thm:main}.) Note that we may assume $\mr{Var}[\Delta]=1$ since $Z_\ell^{(j)}\sim\mc{N}(0,\mr{Var}[\Delta])$ and we are now in a scale-invariant situation with respect to $\Delta$.

We are now purely in a setting of joint convergence of certain bilinear forms of standard Gaussian vectors. Thus, the problem will ultimately reduce to limiting spectral properties of the operators $M_n^\ast$. By a variant of the spectral theorem, since $M_n^\ast$ is skew-symmetric by \cref{lem:operator-decay} (\cref{M2}), we can write $M_n^\ast=Q_n\Sigma_nQ_n^T$ where $Q_n$ is orthogonal and $\Sigma_n$ consists of diagonal $2\times 2$ blocks of the form
\[\begin{bmatrix}0&-\sigma_{n,\ell}\\\sigma_{n,\ell}&0\end{bmatrix}\]
for $\ell\in[\lfloor n/2\rfloor]$, and possibly a single $0$ in the final diagonal entry if $n$ is odd. By orthogonal invariance of Gaussian vectors, applying the orthogonal matrix $Q$, our distribution is the same as
\[\bigg(\frac{G^{(k)T}\Sigma_nG^{(j)}}{n}\bigg)_{1\le j<k\le m}\]
where $G^{(j)}$ are independent standard Gaussian vectors. We have
\begin{equation}\label{eq:gaussian-skew}
\frac{G^{(k)T}\Sigma_nG^{(j)}}{n}=\sum_{\ell=1}^{\lfloor n/2\rfloor}\frac{\sigma_{n,\ell}}{n}(G_{2\ell-1}^{(j)}G_{2\ell}^{(k)}-G_{2\ell}^{(j)}G_{2\ell-1}^{(k)}).
\end{equation}
We have that for any constant $t\ge 1$, $(\sigma_{n,\ell}/n)_{1\le\ell\le t}\to(\sigma_\ell)_{1\le\ell\le t}$ as $n\to\infty$ by \cref{lem:operator-decay} (\cref{M9}).

Now consider some fixed $t\ge 1$ (which we will take to be growing slowly at the end of this argument). Using $\sum_{\ell\ge t}(\sigma_{n,t}/n)^2=O(1/t)$ and Chebyshev's inequality we easily see that with probability $1-O(t^{-1/2})$, the sum in \cref{eq:gaussian-skew} over indices $\ell\ge t$ contributes at most $O(t^{-1/4})$. Furthermore, $(\sigma_{n,\ell}/n)_{1\le\ell\le t}\to(\sigma_\ell)_{1\le\ell\le t}$ as $n\to\infty$ by the above argument. Hence, we deduce that with probability at least $1-O(t^{-1/2})$,
\[\bigg(\frac{G^{(k)T}\Sigma_nG^{(j)}}{n}\bigg)_{1\le j<k\le m}\]
is within $\ell^\infty$ distance $O(t^{-1/4})$ of a random vector which converges to
\[\bigg(\sum_{\ell=1}^t\sigma_\ell(G_{2\ell-1}^{(j)}G_{2\ell}^{(k)}-G_{2\ell}^{(j)}G_{2\ell-1}^{(k)})\bigg)_{1\le j<k\le m}\]
in distribution. Finally, taking $t\to\infty$ slowly gives the desired result, recalling from the remark following \cref{thm:main} that almost surely the appropriate sums converge as $t\to\infty$.
\end{proof}

\section{Properties of coefficient sequences}\label{sec:coefficients}
We next prove \cref{lem:coefficient-combinations}.
\begin{proof}[Proof of \cref{lem:coefficient-combinations}]
By definition we have
\[\sum_{j=1}^n(M_n^\ast)_{jj'}=\sum_{j=1}^n(M_n^\ast)_{j'j}=0\]
hence $\sum_{j=1}^nc_j^{(k^\ast)}=0$ immediately follows, establishing \cref{S2}. Note also that $c_j^{(k^\ast)}$ is a weighted sum of independent random variables $X-1$ where $X\sim\Delta$. Since $\mb{E}[\Delta]=0$ and $\Delta$ is either Poisson or geometric we easily see that it is a sum of independent mean $0$ random variables with bounded $\snorm{X-1}_{\psi_1}$. Additionally, the coefficients of $c_j^{(k^\ast)}$ are of the form $\theta_{kk^\ast}(M_n^\ast)_{jj'}/(2\Theta)$ and $\theta_{k^\ast k}(M_n^\ast)_{j'j})/(2\Theta)$, which by definition and \cref{lem:operator-decay} (\cref{M3}) are bounded in magnitude.

Hence we can apply Bernstein's inequality (\cref{thm:bernstein}) to obtain
\[\mb{P}[|c_j^{(k^\ast)}|\ge t]\le2\exp\Big(-c_{\ref{thm:bernstein}}\min\Big(\frac{t^2}{O(n)},\frac{t}{O(1)}\Big)\Big).\]
Choose $t=\sqrt{n}\log n$, which implies that the event $|c_j^{(k^\ast)}|\ge\sqrt{n}\log n$ occurs with probability at most $\exp(-\Omega((\log n)^2))$. Taking a union bound over $n$ events for $1\le j\le n$, we obtain \cref{S1} with probability $1-n^{-\omega(1)}$.

Finally, \cref{S3} is similar. We wish to show $|c_{j_1}-c_{j_2}|\le\sqrt{|j_1-j_2|}(\log n)^2$ for all $1\le j_1<j_2\le n$ occurs with probability $1-n^{-\omega(1)}$, as then a union bound will finish. To do this, we will exploit cancellation in $(M_n^\ast)_{j_1j'}-(M_n^\ast)_{j_2j'}$. In particular, we have
\begin{align*}
c_{j_1}^{(k^\ast)}-c_{j_2}^{(k^\ast)}&=\frac{1}{2\Theta}\bigg(\sum_{k<k^\ast}\theta_{kk^\ast}\bigg(\sum_{j'=1}^n\Big((M_n^\ast)_{j_1j'}-(M_n^\ast)_{j_2j'}\Big)(X_{j'}^{(k)}-1)\bigg)\\
&\qquad\qquad+\sum_{k>k^\ast}\theta_{k^\ast k}\bigg(\sum_{j'=1}^n\Big((M_n^\ast)_{j'j_1}-(M_n^\ast)_{j'j_2}\Big)(X_{j'}^{(k)}-1)\bigg)\bigg).
\end{align*}
By \cref{lem:operator-decay} (\cref{M7}) we have that $|(M_n^\ast)_{j_1j'}-(M_n^\ast)_{j_2j'}|=O(|j_1-j_2|/n)$ for all but $O(|j_1-j_2|)$ values of $j'$, for which the value is $O(1)$. Since $M_n^\ast$ is skew-symmetric (\cref{M2}), the same occurs when we transpose the matrix. Therefore we can use Bernstein's inequality (\cref{thm:bernstein}) again, this time deducing
\[\mb{P}[|c_{j_1}^{(k^\ast)}-c_{j_2}^{(k^\ast)}|\ge t]\le2\exp\Big(-c_{\ref{thm:bernstein}}\min\Big(\frac{t^2}{O(|j_1-j_2|)},\frac{t}{O(1)}\Big)\Big).\]
Taking $t=\sqrt{|j_1-j_2|}(\log n)^2\ge(\log n)^2$ and taking a union bound, we deduce the desired.
\end{proof}

We now prove that the coefficient sequence coming from \cref{lem:beats} is typically coarse. This will be used to prove \cref{thm:ties} later. We note that the idea of breaking into various intervals and extracting tuples of coefficients with the desired properties also appears in the work of Polymath, in particular in \cite[Lemma~5.10]{Pol22}; however the proofs here are simpler as we require only a ``physical space'' condition on the coefficients.
\begin{lemma}\label{lem:coefficient-sequence}
Let $\Delta\in\{\mr{Geom}(1/2),\mr{Pois}(1)\}$. Let $\wt{X}_j\sim\Delta$ for all $1\le j\le n$ and then condition on $\sum_{j=1}^n\wt{X}_j=n$ and $\sum_{j=1}^nj\wt{X}_j=n(n+1)/2$. If
\[c_j=\sum_{1\le k<j}\wt{X}_j + \frac{\wt{X}_j}{2} - (j-1/2)\]
then with probability $1-n^{-\omega(1)}$ the sequence $(c_j)_{1\le j\le n}$ is coarse (\cref{def:coarse}).
\end{lemma}
\begin{proof}
We will prove that everything but \cref{S2} occurs in the unconditioned independent model with probability $1-n^{-\omega(1)}$. Then note that
\[\mb{P}\bigg[\sum_{j=1}^n\wt{X}_j=n\wedge\sum_{j=1}^nj\wt{X}_j=\frac{n(n+1)}{2}\bigg]\gtrsim n^{-2}\]
from \cref{lem:conditional-mass} (which is proved only using results up to \cref{sec:fourier}) or from the line before \cref{eq:final-levy} in the proof of \cref{thm:main}.

Thus the failure probability of any property in the conditional model will be at most equal to $(n^{-\omega(1)})/(\Omega(n^{-2}))=n^{-\omega(1)}$ by Bayes' rule. So it suffices to consider the independent model, noting that \cref{S2} follows from the conditions $\sum_{j=1}^n\wt{X}_j=n$ and $\sum_{j=1}^nj\wt{X}_j=n(n+1)/2$.

\cref{S1,S3} are simple Bernstein inequality calculations, similar to the proof of \cref{lem:coefficient-combinations}, and we omit the details. For \cref{S5}, note that $c_j=c_{j+1}=c_{j+2}-1/2$ follows if $\wt{X}_j=\wt{X}_{j+1}=0$ and $\wt{X}_{j+2}=1$. Let $J$ be a $3$-separated sequence of size $\Omega(n)$ and note that $j\in J$ satisfies the condition required by \cref{S5} with probability $\Omega(1)$. Thus Bernstein's inequality or Chernoff easily implies \cref{S5}.

For \cref{S6}, consider $J$ which is all multiples of $4y$ in $\{1,\ldots,n-4y\}$, of size $\Omega(n/y)$. For each $j\in J$, the probability that $|c_j-2c_{j+y}+c_{j+2y}|\ge\sqrt{y}$ is seen to be $\Omega(1)$ by the central limit theorem, and this is independent over all $j\in J$. Thus by Bernstein or Chernoff, with probability at least $1-\exp(-\Omega(n/y))$ there are at least $\Omega(n/y)$ many $j\in J$ satisfying the condition required by \cref{S6}. We can repeat the argument for the translations of $J$ by $\{1,2,\ldots,y\}$ and take a union bound, which yields $\Omega(n)$ many indices $j$ with probability $1-n^{-\omega(1)}$ as desired.

Finally, we consider \cref{S4}. We can mimic the proof of \cref{S6} above except with $y=\lfloor n/(\log n)^{3/2}\rfloor$ and still deduce that with probability $1-n^{-\omega(1)}$, there are at least $\Omega(n)$ indices $1\le j\le n-2y$ with $|c_j-2c_{j+y}+c_{j+2y}|\ge\sqrt{y}$. We can pass to a subset $J$ of size $\Omega(n)$ with the property that $j-j'\notin\{\pm y,\pm 2y\}$ for all $j,j'\in J$. For each $j\in J$ we have
\[(c_j-aj-b)^2+(c_{j+y}-a(j+y)-b)^2+(c_{j+2y}-a(j+2y)-b)^2\ge\frac{1}{4}(c_j-2c_{j+y}+c_{j+2y})^2\]
using the inequality $x_1^2+x_2^2+x_3^2\ge(x_1-2x_2+x_3)^2/4$. Hence we deduce
\[\sum_{j=1}^n(c_j-aj-b)^2\gtrsim|J|\cdot(\sqrt{y})^2/4\gtrsim ny\ge n^2/(\log n)^2\]
for all $a,b\in\mb{R}$. The result follows.
\end{proof}

\section{Consequences of \texorpdfstring{\cref{thm:main}}{Theorem 1.4}}\label{sec:consequence}
We now derive the claimed symmetry facts from the statement of \cref{thm:main}.
\begin{proof}[Proof of \cref{cor:symmetry}]
For the first consequence, note that the Gaussian distribution is negation invariant and therefore the result for reversing the edges at vertex $u$ follows by negating the Gaussian $G^{(u)}$ in \cref{thm:main}. For the second consequence, simple replace every die with its ``complement'', i.e., we map $(a_1,\ldots,a_n)$ to $(n+1-a_n,\ldots,n+1-a_i)$. (In the limiting expression of \cref{thm:main}, this corresponds to switching $G_{2\ell-1}^{(j)}$ and $G_{2\ell}^{(j)}$ for all $1\le j\le m$ and $\ell\ge 1$.)
\end{proof}

Now we turn to \cref{cor:convergence}. We require the following lemma relating a tournamenton having image in the set $\{0,1\}$ to the distribution of its $k$-vertex subtournaments.
\begin{lemma}\label{lem:measure-theory}
Fix a tournamenton $\mc{T}$. Suppose that for every $\eps>0$, for all $M$ sufficiently large there is a set $\mc{F}_M$ of $M$-vertex tournaments with $|\mc{F}_M|\le 2^{\eps M^2}$ such that a $\mc{T}$-random tournament on $M$ vertices lies in $\mc{F}_M$ with probability at least $1-\eps$. Then $\mu(\{(x,y)\colon\mc{T}(x,y)\notin\{0,1\}\}) = 0$ where $\mu$ is the Lebesgue measure on $[0,1]^2$.
\end{lemma}
\begin{proof}
Suppose that $\mu(\{(x,y)\colon\mc{T}(x,y)\notin \{0,1\}\})>0$. Then there exists $\delta>0$ such that
\begin{equation}\label{eq:measure-nonzero}
\mu(\{(x,y)\colon\mc{T}(x,y)\in[\delta,1-\delta]\})\ge\delta.
\end{equation}
Consider sampling $M$ random points $x_1,\ldots,x_M$ from $0$ to $1$ uniformly at random. The $\mc{T}$-random tournament is obtained by sampling a directed edge from $x_i$ to $x_j$ with probability $\mc{T}(x_i,x_j)$ (and otherwise putting one from $x_j$ to $x_i$) for all $1\le i<j\le M$. Let 
\[X_M = \{(i,j)\in[M]^2\colon i<j\text{ and }\mc{T}(x_i,x_j)\in[\delta,1-\delta]\}.\]
We have that $\mb{E}|X_M|\ge\delta\binom{M}{2}$ from \cref{eq:measure-nonzero} and thus by applying the Azuma--Hoeffding inequality (\cref{lem:azuma}) on the Doob martingale formed by revealing $x_1,\ldots,x_M$ in order, we see that $\mb{P}[X_M\ge\delta M^2/4]\ge 1-\delta$ for $M$ sufficiently large as a function of $\delta$.

This means there is an event $\mc{E}$ occurring with probability at least $1-\delta$ over the randomness of $x_1,\ldots,x_M$ such that conditional on $\mc{F}$, the entropy of our $\mc{T}$-random tournament is at least $H(\mr{Ber}(\delta))\cdot|X_M|\gtrsim\delta^3M^2$.

But by initial assumption there is an event $\mc{F}$ holding with probability $1-\eps$ such that the original $M$-vertex tournament conditional on $\mc{E}$ is in $\mc{F}_M$. We see that the entropy of the $\mc{T}$-random tournament must be at most $H(\eps)+\log_2|\mc{F}_M|+\eps\log_2(2^{m^2})\lesssim\eps M^2$. Taking $\eps$ much smaller than $\delta^3$ and $M$ sufficiently large, we obtain a contradiction.
\end{proof}

We now are in position to prove \cref{cor:convergence}. 
\begin{proof}[Proof of \cref{cor:convergence}]
We first note that $T_n$ converges to a limit tournamenton $\mc{T}$ since \cref{thm:main} implies that for a fixed digraph $D$ the associated densities converge. Thus the result follows via convergence of subgraph densities implying convergence in cut metric (see \cite{DJ08} where this theory is worked out in the case of directed graphs; the theory for tournamentons follows as a direct consequence via say applying \cite[Theorem~4.1]{Th18} which characterizes a directed graph limit being a tournamenton in terms of certain subgraph counts vanishing).

The more difficult part of \cref{cor:convergence} is verifying the conditions of \cref{lem:measure-theory}. Fix $m$ dice, where we will consider $m$ large, and consider the random series $H_{jk}$ from \cref{thm:main}. For these $m$ dice, reveal $G_\ell^{(j)}$ for $\ell\le 2\lfloor m^{1/2}\rfloor$ and round the value to the nearest $1/m^{25}$, and label each vertex with the corresponding tuple of values. Call the collection of these labels $L(G)$, which depends only on $G_\ell^{(j)}$ for $\ell\le2\lfloor m^{1/2}\rfloor$. Note that with probability $1-\exp(-\Omega(m))$ all these sampled Gaussians are bounded by $m$ and hence there is a set $\mc{L}$ of at most $\exp(O(m^{3/2}\log m))$ different possible labelings such that $L(G)\in\mc{L}$ under this event. Furthermore given these labels $L(G)$, the value
\[H_{jk}^\ast=\sum_{\ell=1}^{\lfloor m^{1/2}\rfloor}\sigma_\ell(G_{2\ell-1}^{(j)}G_{2\ell}^{(k)}-G_{2\ell}^{(j)}G_{2\ell-1}^{(k)})\]
is pinned down to within an interval $I_{jk}(G)$ (defined whenever $L(G)\in\mc{L}$) of length at most $m^{-20}$, say, for all $1\le j<k\le m$.

Note that $H_{j,k}-H_{j,k}^\ast$ has variance $O(m^{-1/2})$ and hence with probability $1-\exp(-m^{-\Omega(1)})$ all these infinite tails are of magnitude at most say $m^{-1/5}$ by \cref{thm:concentration-hypercontractivity}.

Let $\mc{L}'$ be the set of labelings $L(G)\in\mc{L}$ such that the interval $I_{jk}(G)$ intersects $[-m^{-1/5},m^{1/5}]$ for at most $m^{2-1/20}$ many choices of $1\le j<k\le m$. Let $B(G)$ be the set of $(j,k)$ where there is an intersection. Note $B(G)$ depends only on $L(G)$ whenever $L(G)\in\mc{L}$. Combining the observations above, there is an event $\mc{E}$ which occurs with probability $1-\exp(-m^{\Omega(1)})$ such that the following holds if we assume $\mc{E}$:
\begin{itemize}
    \item $L(G)\in\mc{L}$ where $\mc{L}$ is a deterministic set of size $\exp(O(m^{3/2}\log m))$;
    \item If $L(G)\in\mc{L}'$ then the digraph $D(G):=\{(j,k)\colon H_{jk}>0\}$ depends only on the identity of $L(G)$ and on whether $(j,k)$ or $(k,j)$ is in $D$ for all $(j,k)\in B(G)$. Here $\mc{L}'$ is the deterministic subset of $\mc{L}$ defined above.
\end{itemize}
If we can show that $L(G)\in\mc{L}'$ with probability $1-O(m^{-1/20})$, say, then by \cref{thm:main} this will establish the hypotheses of \cref{lem:measure-theory} and hence this will finish the proof. Indeed, then we know that with good probability the digraph $D(G)$ can be determined by revealing $L(G)\in\mc{L}$ (with $\exp(O(m^{3/2}\log m))$ choices), which determines $B(G)$, and then revealing whether $(j,k)\in D(G)$ for all $(j,k)\in B(G)$, which has at most $2^{m^{2-1/20}}$ choices. This will establish the hypothesis of \cref{lem:measure-theory} for $M=\Omega(\eps^{-20})$, say.

Finally, by \cref{thm:gauss-anti} for fixed $1\le j<k\le m$ the probability that $H_{jk}^\ast=O(m^{-1/5})$ is at most $O(m^{-1/10})$. Therefore by Markov's inequality, there are at most $m^{2-1/20}$ pairs in $B(G)$ with probability $1-O(m^{-1/20})$. The result follows.
\end{proof}

We end by providing a short proof given \cref{cor:convergence} that the four-cycle ($A$ beats $B$, $B$ beats $C$, $C$ beats $D$, and $D$ beats $A$) occurs with a greater than $1/16$ limiting probability. This consequence was the method used by Cornacchia and H{\k{a}}z{\l}a to disprove quasirandomness of dice tournaments (in the simpler model where die faces are drawn independently at random from the uniform distribution on $[0,1]$). The surprising fact that this limiting probability is larger than $1/16$ falls out naturally of a \cref{cor:symmetry,cor:convergence}. Note that this means that if $A$ beats $B$, $B$ beats $C$, and $C$ beats $D$ then $D$ is \emph{more} likely to beat $A$ in the limit, since a path with $3$ edges is a tree so has limiting probability $1/8$. (An analogous result holds for larger even cycles which we leave as an exercise to the reader.)
\begin{proposition}\label{prop:4-cycle}
Let $\mc{T}$ be as in \cref{cor:convergence}. We have that 
\[\int_0^1\int_0^1\int_0^1\int_0^1 \mc{T}(x_1,x_2)\mc{T}(x_2,x_3)\mc{T}(x_3,x_4)\mc{T}(x_4,x_1)dx_1dx_2dx_3dx_4>\frac{1}{16}.\]
\end{proposition}
\begin{proof}
Note that 
\begin{align*}
&\int_0^1\int_0^1\int_0^1\int_0^1 \mc{T}(x_1,x_2)\mc{T}(x_2,x_3)\mc{T}(x_3,x_4)\mc{T}(x_4,x_1)dx_1dx_2dx_3dx_4\\
&\qquad= \int_0^1\int_0^1\int_0^1\int_0^1 \mc{T}(x_1,x_2)\mc{T}(x_2,x_3)\mc{T}(x_4,x_3)\mc{T}(x_1,x_4)dx_1dx_2dx_3dx_4\\
&\qquad= \int_0^1\int_0^1\bigg(\int_0^1 \mc{T}(x_1,x_2)\mc{T}(x_2,x_3)dx_2\bigg)^2dx_1dx_3\\
&\qquad\ge\bigg(\int_0^1\int_0^1\int_0^1 \mc{T}(x_1,x_2)\mc{T}(x_2,x_3)dx_1dx_3dx_2\bigg)^2\\
&\qquad=\bigg(\int_0^1\int_0^1\int_0^1 \mc{T}(x_2,x_1)\mc{T}(x_2,x_3)dx_1dx_3dx_2\bigg)^2\\
&\qquad=\bigg(\int_0^1\bigg(\int_0^1 \mc{T}(x_2,x_1)dx_1\bigg)^2dx_2\bigg)^2\\
&\qquad\ge \bigg(\int_0^1\int_0^1\mc{T}(x_2,x_1)dx_1dx_2\bigg)^4=\frac{1}{16}
\end{align*}
where we have applied \cref{cor:symmetry} on vertex $4$, factoring the square, Cauchy--Schwarz, \cref{cor:symmetry} on vertex $1$, factoring the square, Cauchy--Schwarz, and then used that $\mc{T}$ has average $1/2$ (from \cref{thm:main} and symmetry). For equality to occur we must have that for almost all $x,y$, 
\[1/4 = \int_0^1\mc{T}(x,z)\mc{T}(z,y)dz.\]
By the equivalence of codegree counts with quasirandomness for tournaments; see \cite[P4, Theorem~1]{CG91} in work of Chung and Graham, in order for equality to occur we must have $\mc{T}(x,y) = 1/2$ almost everywhere in Lebesgue measure. This contradicts the statement of \cref{cor:convergence}, so the inequality is strict.
\end{proof}

We also briefly derive that any tournament $T$ on $m$ vertices occurs in the limit with positive probability. This recovers a recent result of Akin \cite{Aki21} (which was proven by dynamical methods and which in turn reproves results of Moon and Moser \cite{MM67} which allows for dice to not have the same means and a result of Finkelstein and Thorp \cite{FT00} which constructs cycles of arbitrary length via a more explicit construction).
\begin{proposition}\label{prop:construct}
Recall the setup of \cref{thm:main} and fix any digraph $D$. We have that 
\[\mb{P}[H_{jk}>0\emph{ for all }jk\in E(D)]>0.\]
Equivalently, the $D$-density in $\mc{T}$ is positive. In particular, given any $D$ there exists a set of dice which produce the digraph $D$.
\end{proposition}
\begin{proof}
Let $C$ be a sufficiently large constant to be chosen later. Let $|E(D)|=u$ and label the edges of the digraph $D$ by $e_1,\ldots,e_u$ and the vertices by $1,\ldots,m$. We may assume $D$ is connected so $m\le u+1$. We correspond the indices $\{2\ell-1,2\ell\}$ to directed edge $e_\ell$. For each $\ell\in[u]$ and each $i\in[m]$ which is not an endpoint of the edge $e_\ell$ we define the event $\mc{E}_{\ell,i}$:
\[\max\{|G_{2\ell-1}^{(i)}|,|G_{2\ell}^{(i)}|\}\le 1.\]
For each $\ell\in[|E(D)|]$, if $e_\ell$ is an edge directed from $j$ to $k$ then we define the event $\mc{E}_\ell$:
\[G_{2\ell-1}^{(j)},G_{2\ell}^{(k)}\ge Cu^{1/2},\qquad|G_{2\ell}^{(j)}|,|G_{2\ell-1}^{(k)}|\le 1.\]

Recall that
\[H_{jk}=\sum_{\ell\ge 1}\sigma_\ell(G_{2\ell-1}^{(j)}G_{2\ell}^{(k)}-G_{2\ell}^{(j)}G_{2\ell-1}^{(k)}).\]
Let us further define the event $\mc{E}_{\mr{tail}}$:
\[\bigg|\sum_{\ell>|E(D)|}\sigma_\ell(G_{2\ell-1}^{(j)}G_{2\ell}^{(k)}-G_{2\ell}^{(j)}G_{2\ell-1}^{(k)})\bigg|\le C\]
for all $jk\in E(D)$. By Chebyshev's inequality, for any given $j,k\in[m]$ this occurs with probability $1-O(1/(C^2u))$ Taking a union bound over $jk\in E(D)$ we see $\mb{P}[\mc{E}_{\mr{tail}}]\ge 1/2$ if $C$ is large enough.

Note that $\mc{E}_{\ell,i},\mc{E}_\ell,\mc{E}_{\mr{tail}}$ are jointly independent, so they jointly occur with positive probability (at least $\exp(-\Omega(u^2))$).

If $jk$ is directed edge $e_{\ell^\ast}$, then by construction the first $|E(D)|$ terms of $H_{jk}$ contribute $\gtrsim C^2\sigma_{\ell^\ast}u$, so $\Omega(C^2)$ by \cref{lem:operator-decay} (\cref{M8}). On the other hand, the tail $\ell>|E(D)|$ contributes at most $C$. Thus $H_{jk}>0$ if $C$ was chosen large enough. The result follows.
\end{proof}

\section{Proof of \texorpdfstring{\cref{thm:ties}}{Theorem 1.8}}\label{sec:ties-proof}
Finally we compute the probability of having a tie. We proceed in a slightly indirect manner via first considering the probability that a given \emph{coarse} die (i.e., an appropriate associated sequence is coarse in the sense of \cref{def:coarse}) ties with a randomly sampled die. This amounts to computing the chance that $T_1=T_2=0$ and $T_1=T_2=T_3=0$ given the setup of \cref{def:triple-var-setup}, which will be the first step in understanding the necessary probability.
\begin{lemma}\label{lem:conditional-mass}
Assume the setup of \cref{def:triple-var-setup} and that $(c_j)_{1\le j\le n}$ is coarse. We have
\begin{align*}
\mb{P}[T_1 = 0 \wedge T_2 = 0] &= \frac{\sqrt{3}}{\pi\mr{Var}[\Delta]n^2}+O(n^{-5/2}(\log n)^{34}),\\
\mb{P}[T_1 = 0 \wedge T_2 = 0 \wedge T_3 = 0] &= \frac{\sqrt{3}}{(2\pi\mr{Var}[\Delta])^{3/2}n^2(\min_{a,b\in\mb{R}}\sum_{j=1}^n(c_j-aj-b)^2)^{1/2}}+ O(n^{-7/2}(\log n)^{37}).
\end{align*}
\end{lemma}
\begin{proof}
Let 
\[\wt{T}_1 = \sum_{j=1}^n\wt{X}_j,\quad\wt{T}_2 = \sum_{j=1}^nj\wt{X}_j,\quad\wt{T}_3 = 2\sum_{j=1}^nc_j\wt{X}_j,\]
where $\wt{X}_j\sim\mc{N}(0,\mr{Var}[\Delta])$ are independent Gaussians. Let $\vec{X}\in\mb{R}^n$ be the vector with these coordinates. Define the sets
\begin{align*}
R_2&=\{(\xi_1,\xi_2)\colon|\xi_1|\le n^{-1/2}(\log n)^7,~|\xi_2|\le n^{-3/2}(\log n)^6\},\\
R_3&=\{(\xi_1,\xi_2,\xi_3)\colon|\xi_1|\le n^{-1/2}(\log n)^7,~|\xi_2|\le n^{-3/2}(\log n)^6,~|\xi_3|\le n^{-1}(\log n)^3\}.
\end{align*}
By \cref{thm:inversion-formula} and \cref{lem:theta3-high,lem:theta3-mid,lem:theta2-high,lem:theta2-mid,lem:theta1-region} we easily see
\begin{align*}
\mb{P}[T_1 = 0 \wedge T_2 = 0] &= (2\pi)^{-2}\int_{R_2}\mb{E}[\exp(i\vec{\xi}\cdot(T_1,T_2))]d\vec{\xi}\pm n^{-\omega(1)},\\
\mb{P}[T_1 = 0 \wedge T_2 = 0 \wedge T_3 = 0] &= (2\pi)^{-3}\int_{R_3}\mb{E}[\exp(i\vec{\xi}\cdot(T_1,T_2,T_3))]d\vec{\xi}\pm n^{-\omega(1)}.
\end{align*}
Then, using \cref{lem:clt-region} to transfer to Gaussians we find
\begin{align}\label{eq:conditional-mass-fourier-2d}
\mb{P}[T_1 = 0 \wedge T_2 = 0] &= (2\pi)^{-2}\int_{R_2}\mb{E}[\exp(i\vec{\xi}\cdot(\wt{T}_1,\wt{T}_2))]d\vec{\xi} + O(n^{-5/2}(\log n)^{34}),\\
\mb{P}[T_1 = 0 \wedge T_2 = 0 \wedge T_3 = 0] &= (2\pi)^{-3}\int_{R_3}\mb{E}[\exp(i\vec{\xi}\cdot(\wt{T}_1,\wt{T}_2,\wt{T}_3))]d\vec{\xi} + O(n^{-7/2}(\log n)^{37}).\label{eq:conditional-mass-fourier-3d}
\end{align}

We define the matrix $3\times n$ matrix $M_3$ via
\[M_3 := \begin{pmatrix}
1 & 1 & \ldots & 1\\
1 & 2 & \ldots & n\\
2c_1 & 2c_2 & \ldots & 2c_n
\end{pmatrix}\]
and we let $M_2$ be the first two rows of $M_3$. We have $\vec{\xi}\cdot(\wt{T}_1,\wt{T}_2,\wt{T}_3)=\vec{\xi}^TM_3\vec{X}$ and $\mb{E}[\vec{\xi}\cdot(\wt{T}_1,\wt{T}_2,\wt{T}_3)] = 0$. Furthermore note that $\mb{E}[(\vec{\xi}^TM_3\vec{X})^2] = \mr{Var}[\Delta]\cdot(\vec{\xi}^TM_3M_3^T\vec{\xi})$. We deduce
\begin{align*}
\mb{P}[T_1 = 0 \wedge T_2 = 0 \wedge T_3 = 0] &= (2\pi)^{-3}\int_{R_3}\exp(-\mr{Var}[\Delta]\cdot\xi^TM_3M_3^T\vec{\xi}/2))d\vec{\xi} + O(n^{-7/2}(\log n)^{37}),\\
\mb{P}[T_1 = 0 \wedge T_2 = 0] &= (2\pi)^{-2}\int_{R_2}\exp(-\mr{Var}[\Delta]\cdot\vec{\xi}^TM_2M_2^T\vec{\xi}/2)d\vec{\xi} + O(n^{-5/2}(\log n)^{34}).
\end{align*}
We compute
\[M_3M_3^T = \begin{pmatrix}
n & n(n+1)/2 & 0\\
n(n+1)/2 &  n(n+1)(2n+1)/6 &  2\sum_{j=1}^njc_j\\
0 & 2\sum_{j=1}^njc_j & 4\sum_{j=1}^nc_j^2
\end{pmatrix},\]
recalling \cref{S2} which implies $(1,\ldots,1)\cdot(c_1,\ldots,c_n)=0$. Also, since $M_3^T\vec{e}_1,M_3^T\vec{e}_3$ are orthogonal (denoting $\vec{e}_j\in\mb{R}^3$ as the $j$th elementary vector), we deduce
\begin{align*}
\mr{dist}(M^T\vec{e}_2,&\mr{span}_{\mb{R}}\{M^T\vec{e}_j\}_{j\in\{1,3\}})^2\\
&= \mr{dist}((-(n-1)/2,-(n-3)/2,\ldots, (n-1)/2), \mr{span}_{\mb{R}}\{(c_1,\ldots,c_n)\})^2 \\
&\gtrsim n^3\bigg(1-\frac{\sang{(-(n-1)/2,-(n-3)/2,\ldots,(n-1)/2), (c_1,\ldots,c_n)}^2}{\snorm{(-(n-1)/2,-(n-3)/2,\ldots, (n-1)/2)}_2^2\snorm{(c_1,\ldots,c_n)}_2^2}\bigg)\\
&\gtrsim n^3\cdot \frac{\min_{a,b\in\mb{R}}\sum_{j=1}^n(c_j-aj-b)^2}{\sum_{j=1}^nc_j^2}\\
&\gtrsim n^3/(\log n)^4.
\end{align*}
The second-to-last line comes from noting that the desired minimum corresponds to the distance from $(c_1,\ldots,c_n)$ to the plane spanned by $M^T\vec{e}_1,M^T\vec{e}_2$. The last line uses \cref{S1,S4}. Similarly, we find
\[\mr{dist}(M^T\vec{e}_3,\mr{span}_{\mb{R}}(\{M^T\vec{e}_j\}_{j\in\{1,2\}})^2 = \min_{a,b\in\mb{R}}\sum_{j=1}^n(c_j-aj-b)^2\ge n^2/(\log n)^2.\]
Therefore we have
\begin{align*}
\xi^TM_3M_3^T\xi & = \snorm{M_3^T\xi}_2^2\ge\max_{j^\ast\in[3]}\xi_{j^\ast}^2\mr{dist}(M^T\vec{e}_{j^\ast},\mr{span}_{\mb{R}}(\{M\vec{e}_j\}_{j\in [3]\setminus\{j^\ast\}})^2\\
&\gtrsim n\xi_1^2 + n^3\xi_2^2/(\log n)^4 + n^2\xi_3^2/(\log n)^2.
\end{align*}

This inequality immediately allows us to extend the regions of integration in \cref{eq:conditional-mass-fourier-2d,eq:conditional-mass-fourier-3d} to $\mb{R}^2$ and $\mb{R}^3$, respectively, since the integrand within the remaining region is super-polynomially small and decaying rapidly. Applying the formula for a Gaussian integral, we have
\begin{align*}
\mb{P}[T_1 = 0 \wedge T_2 = 0] &= (2\pi)^{-2}\int_{\mb{R}^2}\exp(-\mr{Var}[\Delta]\cdot\vec{\xi}^TM_2M_2^T\vec{\xi}/2)d\vec{\xi} + O(n^{-5/2}(\log n)^{34})\\
&= (2\pi)^{-1}(\det M_2M_2^T)^{-1/2}+ O(n^{-5/2}(\log n)^{34})\\
& = \frac{\sqrt{3}}{\pi\mr{Var}[\Delta]n^2} + O(n^{-5/2}(\log n)^{34}).
\end{align*}
We computed the determinant explicitly as $n^2(n^2-1)/12$ using the expression for $M_3M_3^T$. Similarly, we have
\begin{align*}
\mb{P}[T_1 = 0 \wedge T_2 = 0 \wedge T_3 = 0] &= (2\pi)^{-3}\int_{\mb{R}^3}\exp(-\mr{Var}[\Delta]\cdot\vec{\xi}^TM_3M_3^T\vec{\xi}/2)d\vec{x} + O(n^{-7/2}(\log n)^{37})\\
&= (2\pi\mr{Var}[\Delta])^{-3/2}\det(M_3M_3^T)^{-1/2}+ O(n^{-7/2}(\log n)^{37})\\
& = \frac{\sqrt{3}}{(2\pi\mr{Var}[\Delta])^{3/2}n^2(\min_{a,b\in\mb{R}}\sum_{j=1}^n(c_j-aj-b)^2)^{1/2}}+ O(n^{-7/2}(\log n)^{37}).
\end{align*}
In the last line we used the base times height formula row-by-row to compute $(\det M_3M_3^T)^{1/2}$, which can be interpreted as the $3$-dimensional volume of the corresponding parallelepiped spanned by $M_3^T\vec{e}_1,M_3^T\vec{e}_2,M_3^T\vec{e}_3$ within $\mb{R}^n$.
\end{proof}

Given \cref{lem:conditional-mass}, and recalling \cref{lem:crucial}, the approach will now be to take an average of $\mb{P}[T_3=0|T_1=T_2=0]$ over the distribution of coarse sequences $(c_j)_{1\le j\le n}$ that come from the frequency count statistics of a typical die sampled from either model. This requires us to understand the quadratic expressions $\min_{a,b\in\mb{R}}\sum_{j=1}^n(c_j-aj-b)^2$ (where $c_j$ will linearly depend on the frequency count statistics) conditional on stuff such as $T_1=T_2=0$. At a high level, we will reduce understanding a quadratic form to understanding finitely many linear forms jointly (via sampling random rows to take a dot product against; heuristically, one could study the large singular vectors). Thus we will reduce to a situation where we only need the sort of Fourier control guaranteed by \cref{sec:fourier}. Specifically, we can prove \cref{lem:fourier-coeff} given the tools in \cref{sec:fourier}, which is the key estimate. Beyond this, we need various tools to control certain tail estimates and related notions that occur in the course of the proof, which is very much related to the fact that evaluating $\mb{E}(\sum_{\ell\ge 1}\sigma_\ell^2(Z_\ell^2+Z_\ell'^2))^{-1/2}$ involves a (convergent) improper integral.

We will require the following estimate regarding sampling independent points for a given distribution on $[n]$. We use the following estimate on sums of independent random variables from \cite{vE65}.
\begin{theorem}[{\cite[Theorem~4]{vE65}}]\label{thm:estimate-good}
Fix $\beta\in [1,2]$. There exists $C_{\ref{thm:estimate-good}}(\beta)>0$ such that the following holds. Let $X_i$ be independent mean zero random variables with $\mb{E}[|X_i|^\beta]<\infty$. We have that 
\[\mb{E}[|\sum_{i=1}^{n}X_i|^{\beta}]\le C_{\ref{thm:estimate-good}}(\beta)\sum_{i=1}^{n}\mb{E}[|X_i|^{\beta}]\]
\end{theorem}

\begin{lemma}\label{lem:resample-estimate}
There exists a constant $C_{\ref{lem:resample-estimate}}>0$ such that the following holds. Given a sequence $x_1,\ldots,x_n$, let $i_1,\ldots,i_M$ be indices chosen uniformly at random from $[n]$. Then 
\[\bigg|\frac{1}{M}\sum_{j=1}^Mx_{i_j} - \frac{1}{n}\sum_{j=1}^nx_j\bigg|\le C_{\ref{lem:resample-estimate}}M^{-1/4}\bigg(\frac{\sum_{j=1}^n|x_j|^{3/2}}{n}\bigg)^{2/3}\]
occurs with probability at least $1-M^{-1/8}$.
\end{lemma}
\begin{proof}
We have that 
\begin{align*}
\mb{E}\bigg|\frac{1}{M}\sum_{j=1}^Mx_{i_j} - \frac{1}{n}\sum_{j=1}^nx_j\bigg|^{3/2}& = M^{-3/2} \mb{E}\bigg|\sum_{j=1}^M\bigg(x_{i_j} - \frac{1}{n}\sum_{j=1}^nx_j\bigg)\bigg|^{3/2}\\
&\lesssim M^{-1/2}\cdot \mb{E}\bigg|x_{i_1} - \frac{1}{n}\sum_{j=1}^nx_j\bigg|^{3/2}\\
&\lesssim M^{-1/2}\cdot \mb{E}|x_{i_1} - x_{i_2}|^{3/2}\lesssim M^{-1/2}\cdot \mb{E}|x_{i_1}|^{3/2} \\
&= n^{-1}M^{-1/2}\sum_{i=1}^{n}|x_i|^{3/2}
\end{align*}
where we have used \cref{thm:estimate-good}, Jensen's inequality, and that $|x-y|^{3/2}\le 2^{1/2}(|x|^{3/2}+|y|^{3/2})$. The desired follows immediately by Markov's inequality.
\end{proof}

We next verify the following key distributional identity. We will use it after applying the identity $\min_{a,b\in\mb{R}}\sum_{j}(c_j-a-bj)^2 = \snorm{M_n^\ast\vec{c}}_2^2$ where $\vec{c}=(c_1,\ldots,c_n)$.
\begin{lemma}\label{lem:fourier-coeff}
Let $X_j$, $T_1,T_2$ be as in \cref{def:triple-var-setup}, sample $\wt{X}_j\sim\mc{N}(0,\mr{Var}[\Delta])$ for $1\le j\le n$, let $X = (X_1,\ldots,X_n)$, and let $\wt{X} = (\wt{X}_1,\ldots,\wt{X}_n)$. If $M\le\log\log n$ and $j_k\in[n]$ and $|\Theta_k|\le(\log n)^{7/4}$ for $1\le k\le M$ then
\[\bigg|\mb{E}\bigg[\exp\bigg(i\sum_{k=1}^M\frac{\Theta_k\sang{X,M_n^\ast\vec{e}_{j_k}}}{\sqrt{n}}\bigg)\bigg|T_1 = T_2 = 0\bigg] - \mb{E}\bigg[\exp\bigg(i\sum_{k=1}^M\frac{\Theta_k\sang{\wt{X},M_n^\ast\vec{e}_{j_k}}}{\sqrt{n}}\bigg)\bigg]\bigg|\le n^{-1/2}(\log n)^{39}.\]  
\end{lemma}
\begin{proof}
The proof is very similar to the first part of the proof of \cref{thm:main} as well as the proof of \cref{lem:conditional-mass}, so we will be brief and focus only on the necessary modifications from the basic proof strategy.

We apply \cref{thm:inversion-formula} to deduce
\[\mb{E}\bigg[\mbm{1}_{T_1=T_2=0}\exp\bigg(i\sum_{k=1}^M\frac{\Theta_k\sang{X,M_n^\ast\vec{e}_{j_k}}}{\sqrt{n}}\bigg)\bigg]=\frac{1}{2\pi}\int_{[-\pi,\pi]^2}\mb{E}\exp\bigg(i\sum_{k=1}^M\frac{\Theta_k\sang{X,M_n^\ast\vec{e}_{j_k}}}{\sqrt{n}}+i\vec{\xi}\cdot(T_1,T_2)\bigg)d\vec{\xi}\]
and then apply \cref{lem:theta2-high,lem:theta2-mid,lem:theta1-region}. In order to apply these, we define
\[c_j^\ast:=n\sum_{k=1}^M\frac{\Theta_k(M_n^\ast \vec{e}_{j_k})_j}{2\sqrt{n}}\]
so that
\[\sum_{k=1}^M\frac{\Theta_k\sang{X,M_n^\ast\vec{e}_{j_k}}}{\sqrt{n}}=\frac{2\sum_{j=1}^nc_j^\ast X_j}{n}.\]
However, $(c_j^\ast)_{1\le j\le n}$ does not quite satisfy \cref{S3} so we cannot apply the lemmas directly; the only obstruction is that $c_j^\ast$ will have slight ``local jumps'' near $j\in\{j_1,\ldots,j_M\}$ due to the fact that $M_n^\ast$ has slightly ``discontinuous'' entries along the diagonal. Indeed, one can check
\[|c_j^\ast-c_{j'}^\ast|\lesssim\frac{|j-j'|}{\sqrt{n}}(\log n)^{7/4}\log\log n+\sum_{k=1}^M\mbm{1}_{j_k\in[j,j']}\]
for all $1\le j<j'\le n$ due to \cref{M7}. This is at most $\sqrt{|j-j'|}(\log n)^2$ if $\{j_1,\ldots,j_M\}\cap[j,j']=\emptyset$.

To fix this obstruction, we simply apply \cref{lem:theta2-high,lem:theta2-mid,lem:theta1-region} to a consecutive sequence of entries. Note that there is some $\{n_1+1,n_1+2,\ldots,n_2\}\subseteq[n]$ with $|n_2-n_1|\ge n/(2M)$ with no $j_k$ contained in this consecutive range. Let $n' = |n_2-n_1|$. Therefore, we may condition on values of $X_{[n]\setminus\{n_1+1,\ldots,n_2\}}$ and then deduce the necessary estimate from the randomness of $X_{\{n_1+1,\ldots,n_2\}}$. We deduce
\[\bigg|\mb{E}\exp\bigg(i\sum_{k=1}^M\frac{\Theta_k\sang{X,M_n^\ast\vec{e}_{j_k}}}{\sqrt{n}}+i\vec{\xi}\cdot(T_1,T_2)\bigg)\bigg|\le n^{-\omega(1)}\]
as long as $|\xi_2|\in[(n')^{-3/2}(\log n')^6,\pi]$ or $|\xi_2|\le (n')^{-3/2}(\log n')^6$ and $|\xi_1 + n_1\xi_2|\in [(n')^{-1/2}(\log n')^7,5\pi/4]$ (where one applies \cref{lem:theta2-high,lem:theta2-mid,lem:theta1-region}). This trivially covers all $(\xi_1,\xi_2)$ except for say $|\xi_1|\le n^{-1/2}(\log n)^8$ and $|\xi_2|\le n^{-3/2}(\log n)^7$.

Then we can apply \cref{lem:clt-region} (technically, since $\xi_1,\xi_2$ could be slightly larger than the range considered in \cref{lem:clt-region}, it is slightly different but the exact same technique applies). Overall, we deduce
\begin{align*}
&\mb{E}\bigg[\mbm{1}_{T_1=T_2=0}\exp\bigg(i\sum_{k=1}^M\frac{\Theta_k\sang{X,M_n^\ast\vec{e}_{j_k}}}{\sqrt{n}}\bigg)\bigg]\\
&\qquad=\frac{1}{2\pi}\int_{-\tau_1}^{\tau_1}\int_{-\tau_2}^{\tau_2}\mb{E}\exp\bigg(i\sum_{k=1}^M\frac{\Theta_k\sang{\wt{X},M_n^\ast\vec{e}_{j_k}}}{\sqrt{n}}+i\vec{\xi}\cdot(\wt{T}_1,\wt{T}_2)\bigg)d\xi_2d\xi_1 + O(n^{-5/2}(\log n)^{39})
\end{align*}
where $\tau_1=n^{-1/2}(\log n)^8$ and $\tau_2=n^{-3/2}(\log n)^7$ and $\wt{T}_1=\sum_{j=1}^n\wt{X}_j$ and $\wt{T}_2=\sum_{j=1}^nj\wt{X}_j$.

Next, use that $\sum_{j=1}^nc_j^\ast=\sum_{j=1}^njc_j^\ast=0$, which follows from \cref{def:dice-matrix}. This implies that $(\wt{T}_1,\wt{T}_2)$ is independent of the first part of the sum in the exponential. So, we can factor and integrate over $\xi_1,\xi_2$ to deduce
\[\mb{E}\bigg[\mbm{1}_{T_1=T_2=0}\exp\bigg(i\sum_{k=1}^M\frac{\Theta_k\sang{X,M_n^\ast\vec{e}_{j_k}}}{\sqrt{n}}\bigg)\bigg]=q\mb{E}\exp\bigg(i\sum_{k=1}^M\frac{\Theta_k\sang{\wt{X},M_n^\ast\vec{e}_{j_k}}}{\sqrt{n}}\bigg)+O(n^{-5/2}(\log n)^{39})\]
where $q=(2\pi)^{-1}\int_{\mb{R}^2}\mb{E}\exp(i\vec{\xi}\cdot(\wt{T}_1,\wt{T}_2))d\vec{\xi}$ (note completing the integral to $\infty$ does not change the error term). Comparing with the proof of \cref{lem:conditional-mass}, we easily deduce that $q=(1+O(n^{-1/2}(\log n)^{34}))\mb{P}[T_1=0\wedge T_2=0]$ and thus also $q=\Omega(n^{-2})$. Dividing by $\mb{P}[T_1=0\wedge T_2=0]$ and subtracting, we deduce the desired result.
\end{proof}

We deduce an appropriate bound on a moment of $n^{-1/2}|\sang{\vec{x},M_n^\ast \vec{e}_j}|$.
\begin{lemma}\label{lem:upper-tail}
Let $X_j$, $T_1, T_2$ be as in \cref{def:triple-var-setup} and let $X = (X_1,\ldots,X_n)$. We have
\[\mb{E}\bigg[\sum_{j=1}^n\bigg(\frac{|\sang{X,M_n^\ast \vec{e}_j}|}{\sqrt{n}}\bigg)^{3}\bigg|T_1 = T_2 = 0\bigg]\le C_{\ref{lem:upper-tail}}n.\]
Similarly if $X_j'$ is as in \cref{lem:clt-region}, we have 
\[\mb{E}\bigg[\sum_{j=1}^n\bigg(\frac{|\sang{X',M_n^\ast \vec{e}_j}|}{\sqrt{n}}\bigg)^3\bigg]\le C_{\ref{lem:upper-tail}}n.\]
\end{lemma}
\begin{proof}
The second estimate is trivial by linearity of expectation and \cref{lem:operator-decay} (\cref{M3}). For the first estimate, using linearity of expectation it suffices to show $\mb{E}|\sang{X,M_n^\ast \vec{e}_j}|^3=O(n^{3/2})$ uniformly for all $1\le j\le n$.

Note that $|\sang{X,M_n^\ast \vec{e}_j}|\ge\sqrt{n\log n}\log\log n$ occurs with probability $n^{-\omega(1)}$ in the independent model by Bernstein's inequality (\cref{thm:bernstein}), hence since $\mb{P}[T_1=T_2=0]=\Omega(n^{-2})$ from \cref{lem:conditional-mass} we have the same in the conditional model. The tail bound from \cref{thm:bernstein} is good enough that we can in fact obtain
\[\mb{E}|\sang{X,M_n^\ast \vec{e}_j}|^3\mbm{1}_{|\sang{X,M_n^\ast \vec{e}_j}|\ge\sqrt{n\log n}\log\log n}=n^{-\omega(1)}.\]
Now it suffices to consider ``reasonable'' scales for $|\sang{X,M_n^\ast \vec{e}_j}|$.

We apply \cref{thm:esseen} with $\sang{X,M_n^\ast\vec{e}_j}/\sqrt{n}$ conditional on $T_1=T_2=0$ and with $\sang{X',M_n^\ast \vec{e}_j}/\sqrt{n}$. Note that \cref{lem:fourier-coeff} shows an error of $O(n^{-1/2}(\log n)^{39})$ between the two resulting Fourier coefficients, and we deduce
\begin{align*}
&\sup_{\tau\in\mb{R}}|\mb{P}[\sang{X,M_n^\ast \vec{e}_j}\le\tau\sqrt{n}]-\mb{P}[\sang{X',M_n^\ast \vec{e}_j}\le\tau\sqrt{n}]|\\
&\qquad\lesssim\int_{-L}^L\frac{\min\{n^{-1/2}(\log n)^{39},\mb{E}[|t||\sang{X,M_n^\ast \vec{e}_j}-\sang{X',M_n^\ast \vec{e}_j}|]\}}{|t|}dt+1/L\\
&\qquad\lesssim\int_{-L}^L\frac{\min\{n^{-1/2}(\log n)^{39},|t|n^2\}}{|t|}dt+1/L\lesssim n^{-1/2}(\log n)^{40}+1/L\lesssim 1/L
\end{align*}
for $L=(\log n)^{7/4}$. We thus have
\begin{align*}
\mb{E}|\sang{X,M_n^\ast \vec{e}_j}|^3\mbm{1}_{|\sang{X,M_n^\ast \vec{e}_j}|<\sqrt{n\log n}\log\log n}&=\mb{E}|\sang{X',M_n^\ast \vec{e}_j}|^3\mbm{1}_{|\sang{X',M_n^\ast \vec{e}_j}|<\sqrt{n\log n}\log\log n}\\
&\qquad\qquad+O((\sqrt{n\log n}\log\log n)^3/L)=O(n^{3/2})
\end{align*}
by integration by parts.
\end{proof}

We will also require the following variant of \cref{lem:upper-tail} which bounds the difference between nearby coordinates.
\begin{lemma}\label{lem:round}
Let $X_j$, $T_1, T_2$ be as in \cref{def:triple-var-setup} and let $X = (X_1,\ldots,X_n)$. If $\eps \ge 1/\log\log n$, and $|j-j'|\le \eps n$ then we have
\[\mb{E}\bigg[\frac{||\sang{X,M_n^\ast \vec{e}_j}|^2-|\sang{X,M_n^\ast \vec{e}_{j'}}|^2|}{n}\bigg|T_1 = T_2 = 0\bigg]\le C_{\ref{lem:upper-tail}}\eps^{1/2}.\]
Similarly if $X_j'$ is as in \cref{lem:clt-region}, we have 
\[\mb{E}\bigg[\frac{||\sang{X',M_n^\ast \vec{e}_j}|^2-|\sang{X',M_n^\ast \vec{e}_{j'}}|^2|}{n}\bigg]\le C_{\ref{lem:upper-tail}}n.\]
\end{lemma}
\begin{proof}
As $|y^2-z^2| = |y-z|\cdot |y+z|\le |y-z| \cdot (|y|+|z|)$, by Cauchy--Schwarz it suffices to prove that 
\[\mb{E}\bigg[\frac{|\sang{X,M_n^\ast \vec{e}_j}|^2}{n}\bigg|T_1 = T_2 = 0\bigg]\le C_{\ref{lem:upper-tail}},\qquad\mb{E}\bigg[\frac{|\sang{X,M_n^\ast \vec{e}_j}-\sang{X,M_n^\ast \vec{e}_{j'}}|^2}{n}\bigg|T_1 = T_2 = 0\bigg]\le C_{\ref{lem:upper-tail}}\eps\]
and analogous estimates for $X'$. The two estimates follow immediately for $X'$ by \cref{lem:operator-decay} (in particular \cref{M7} for the second estimate). For $X$, note that by H\"older's inequality a strictly stronger estimate than the first is proven in \cref{lem:upper-tail}. For the second estimate, note that if $|j-j'|\le n/((\log n)(\log\log n)^3)$, we have from Bernstein's inequality (\cref{thm:bernstein}):
\[\mb{E}[|\sang{X,M_n^\ast \vec{e}_j-\vec{e}_{j'}}|^2\mbm{1}_{|\sang{X,M_n^\ast \vec{e}_j-\vec{e}_{j'}}|\ge\eps n}]\le n^{-\omega(1)}.\]
Since $\mb{P}[T_1=T_2=0]=\Omega(n^{-2})$ from \cref{lem:conditional-mass} we have the same in the conditional model. For the remaining values of $j$ and $j'$, from the proof technique in \cref{lem:upper-tail} we have 
\begin{align*}
&\sup_{\tau\in\mb{R}}|\mb{P}[\sang{X,M_n^\ast (\vec{e}_j-\vec{e}_{j'})}\le\tau\sqrt{n}]-\mb{P}[\sang{X',M_n^\ast (\vec{e}_j-\vec{e}_{j'})}\le\tau\sqrt{n}]|\lesssim \frac{\sqrt{n}}{T\sqrt{j-j'}}
\end{align*}
for $T=(\log n)^{7/4}$, using $\snorm{M_n^\ast (\vec{e}_j-\vec{e}_{j'})}_2\asymp\sqrt{j-j'}$. By Bernstein's inequality (\cref{thm:bernstein}) we have $\mbm{E}[|\sang{X,M_n^\ast (\vec{e}_j-\vec{e}_{j'})}|^2\mbm{1}_{|\sang{X,M_n^\ast (\vec{e}_j-\vec{e}_{j'})}|\ge \sqrt{(j-j')\log n}(\log\log n)}] = n^{-\omega(1)}$ and therefore the same holds in the conditional model. Thus, similar to the proof of \cref{lem:upper-tail}, we can use this to cut off the values, transfer to the Gaussian model using integration by parts, and bound the resulting expressions. The desired follows.
\end{proof}

We next deduce that $\min_{a,b\in \mb{R}}\sum_{j}(c_j-a-bj)^2$ satisfies an appropriate anticoncentration bound near $0$ with high probability, so as to control singularity behavior.
\begin{lemma}\label{lem:lower-bound}
Let $X_j$, $T_1, T_2$ be as in \cref{def:triple-var-setup} and let $\eps \ge 1/\log n$. Let $X=(X_1,\ldots,X_n)$. We have 
\[\mb{P}\bigg[\sum_{j=1}^n\sang{X,M_n^\ast \vec{e}_j}^2\le \eps n^2\bigg|T_1= T_2=0\bigg]\lesssim\eps^4.\]
\end{lemma}
\begin{proof}
Let $k$ be a sufficiently large absolute integer constant. For $1\le t\le k$ define $S_t = [tn/(k+1), tn/(k+1)  + \delta^2n]$, $S_t' = [tn/(k+1)+ 2\delta^2n, tn/(k+1)  + 3\delta^2n]$ where $\delta\in(0,1/2)$ will be a sufficiently small constant (with respect to $k$) to be chosen later. Let
\[\mc{I} = \{(j_1,j_1',j_2,j_2',\ldots,j_k,j_k')|\colon j_t\in S_t, j_t'\in S_t', |\sang{X,M_n^\ast(\vec{e}_{j_t'}-\vec{e}_{j_t})}|\le \delta^{-2}\eps^{1/2}n^{1/2}\}.\]
If $\sum_{j=1}^n\sang{X,M_n^\ast\vec{e}_j}^2\le \eps n^2$, by Markov's inequality there are fewer than $\delta^4n$ indices $j$ such that $|\sang{\vec{x},M_n^\ast\vec{e}_j}|\ge\delta^{-2}\eps^{1/2}n^{1/2}$. Therefore it follows that $|\mc{I}|\ge (\delta^2n/2)^{2k}$ under this event. 

We now compute 
\begin{align}
\mb{E}[|\mc{I}||T_1= T_2 = 0] &\le\sum_{\substack{j_t\in S_t, j_t'\in S_t'\\\forall t\in[k]}} \mb{P}\bigg[\bigcap_{t = 1}^{k}|\sang{X,M_n^\ast(\vec{e}_{j_t'}-\vec{e}_{j_t})}|\le \delta^{-2}\eps^{1/2}n^{1/2}\bigg]\notag\\
&\lesssim_{k,\delta}\eps^{k/2}n^{2k}\label{eq:moment-size}
\end{align}
if one can prove for any choices of $j_t\in S_t, j_t'\in S_t'$ for $1\le t\le k$ that 
\begin{equation}\label{eq:moment-reduction}
\mb{P}\bigg[\bigcap_{t=1}^k|\sang{X,M_n^\ast(\vec{e}_{j_t'}-\vec{e}_{j_t})}|\le \delta^{-2}\eps^{1/2}n^{1/2}\bigg]\lesssim_{k,\delta} \eps^{k/2}.
\end{equation}
Note \cref{eq:moment-size} immediately implies the desired result taking $k = 8$ and applying Markov's inequality: we find that $|\mc{I}|\ge(\delta^2n/2)^{2k}$ occurs with probability $O(\eps^4)$, which implies the same for the original event by the earlier analysis.

To prove \cref{eq:moment-reduction} the idea is to use \cref{thm:esseen,lem:fourier-coeff}. Writing $B$ for the radius $k\delta^2/\eps^{1/2}$ unit ball in $\mb{R}^k$, we have
\begin{align*}
\mb{P}\bigg[\bigcap_{t=1}^k&\frac{|\sang{X,M_n^\ast(\vec{e}_{j_t'}-\vec{e}_{j_t})}|}{\sqrt{n}}\le\delta^{-2}\eps^{1/2}\bigg]\\
&\lesssim_{\delta,k}\eps^{k/2}\int_B\bigg|\mb{E}\exp\bigg(2\pi i\sum_{t=1}^k\frac{\xi_t\sang{X,M_n^\ast(\vec{e}_{j_t'}-\vec{e}_{j_t})}}{\sqrt{n}}\bigg)\bigg|d\vec{\xi}\\
&\lesssim_{\delta,k}\eps^{k/2}\int_B\bigg|\mb{E}\exp\bigg(2\pi i\sum_{t=1}^k\frac{\xi_t\sang{\wt{X},M_n^\ast(\vec{e}_{j_t'}-\vec{e}_{j_t})}}{\sqrt{n}}\bigg)\bigg|d\vec{\xi}+O(n^{-1/2}(\log n)^{39})\\
&\lesssim_{\delta,k}\eps^{k/2}q+O(n^{-1/2}(\log n)^{39})
\end{align*}
where $q$ is the probability density function of the Gaussian vector $(n^{-1/2}\sang{\wt{X},M_n^\ast(\vec{e}_{j_t'}-\vec{e}_{j_t})})_{1\le t\le k}$ evaluated at $0$. For the last line, we used the nonnegativity of Gaussian characteristic functions and Fourier inversion. Now we show $q=O_{\delta,k}(1)$ to finish.

Note that by \cref{M1} and explicit computation we have
\[M_n^\ast(\vec{e}_{j_t'}-\vec{e}_{j_t}) = \frac{1}{2}\sum_{j_t<k<j_t'}\vec{e}_k + \vec{v}_{j_t,j_t'}\]
where $\snorm{\vec{v}_{j_t,j_t'}}_2\lesssim\delta^2n^{1/2}$ and note that $\snorm{\sum_{j_t<k<j_t'}\vec{e}_k}_2\gtrsim\delta n^{1/2}$. Therefore we find that 
\[\mr{dist}(M_n^\ast(\vec{e}_{j_t'}-\vec{e}_{j_t}), \mr{span}_\mb{R}\{M_n^\ast(\vec{e}_{j_s'}-\vec{e}_{j_s})\colon s\in[k]\setminus\{t\}\})\ge \snorm{M_n^\ast(\vec{e}_{j_t'}-\vec{e}_{j_t})}_2/2\]
if $\delta$ is sufficiently small as a function of $k$. This implies that the covariance matrix of the above Gaussian vector is diagonally dominated with constant order diagonal entries, and the result follows.
\end{proof}

We now conclude with the proof of \cref{thm:ties}.
\begin{proof}[Proof of \cref{thm:ties}]
Given a die $B$ with the frequency counts $(\wt{b}_j)_{1\le j\le n}$, let $y_j=\sum_{1\le k<j}\wt{b}_k +\wt{b}_j/2-(j-1/2)$. We say $B$ is \emph{suitable} if $(y_j)_{1\le j\le n}$ is coarse (\cref{def:coarse}). By \cref{lem:coefficient-sequence} we have that $B$ is coarse with probability $1-n^{-\omega(1)}$. Furthermore, write $\wt{b} = (\wt{b}_1,\ldots,\wt{b}_n)$ and note
\begin{equation}\label{eq:ols-R-score}
\min_{a,b\in\mb{R}}\sum_{j=1}^n(y_j-aj-b)^2 = \snorm{M_n^\ast\wt{b}}_2^2,
\end{equation}
which can be seen by looking at projections of vectors in $\mb{R}^n$, and also noting that $\sum_{j=1}^n(\wt{b}_j-1)=\sum_{j=1}^nj(\wt{b}_j-1)=0$.

Recall that by \cref{lem:crucial}, $(\wt{b}_1,\ldots,\wt{b}_n)$ has the distribution of $(X_1,\ldots,X_n)$ conditional on $T_1=T_2=0$, borrowing the setup of \cref{def:triple-var-setup} (where $\Delta=\mr{Geom}(1/2)$ for the multiset model and $\Delta=\mr{Pois}(1)$ for the balanced sequence model).

Therefore, applying \cref{lem:conditional-mass} we find
\begin{align*}
\mb{P}[A\text{ ties } B] &= \mb{P}[A\text{ ties }B \wedge B\text{ is suitable}] \pm n^{-\omega(1)}\\
&=\mb{E}_B\big[\mb{P}[A\text{ ties }B|\text{suitable }B]\big] \pm n^{-\omega(1)}\\
&= \mb{E}\bigg[\frac{\mbm{1}_{B\text{ is suitable}}}{(8\pi\mr{Var}[\Delta])^{1/2}\snorm{M_n^\ast\wt{b}}_2}\bigg] \pm n^{-3/2}(\log n)^{40}.
\end{align*}
The final line follows from using \cref{lem:crucial} to interpret the probability of a tie as the ratio of the two expressions in \cref{lem:conditional-mass}, and using \cref{S4} to control the resulting error terms.

Fix a constant $\eps > 0$. We will take $\eps\to 0^+$ sufficiently slowly at the end of the proof. We have
\begin{align*}
\mb{E}\bigg[\frac{\mbm{1}_{B\text{ is suitable}}}{\snorm{M_n^\ast\wt{b}}_2} - \frac{\mbm{1}_{B\text{ is suitable}}}{\max(\snorm{M_n^\ast\wt{b}}_2, \eps n)}\bigg] &\lesssim n^{-\omega(1)} + \sum_{(\log n)^{-2}\le 2^{-j} \le \eps} \frac{2^{j}}{n}\cdot \mb{P}[\snorm{M_n^\ast\wt{b}}_2\le 2^{-j}n]\\
&\lesssim n^{-\omega(1)} + \sum_{(\log n)^{-2}\le 2^{-j} \le \eps} \frac{2^j}{n}\cdot 4^{-j}\\
&\lesssim\frac{\eps}{n}
\end{align*}
where we have dyadically decomposed the small values of $\snorm{M_n^\ast\wt{b}}_2$ and applied \cref{lem:lower-bound} (we apply the lemma for $\max(2^{-j},(\log n)^{-1})$). Note that $\sang{X,M_n^\ast\vec{e}_j}=-\sang{M_n^\ast X,\vec{e}_j}$ is the negative of the $j$th coordinate of $M_n^\ast X$ by \cref{M2}, and we again used \cref{lem:crucial}. Additionally, we are using that $\snorm{M_n^\ast\vec{b}}_2\ge n(\log n)^{-2}$ for coarse $\vec{b}$ by \cref{eq:ols-R-score,S4}.

Therefore
\begin{align*}
(8\pi\mr{Var}[\Delta])^{1/2}\mb{P}[A\text{ ties } B] &= \mb{E}\bigg[\frac{\mbm{1}_{B\text{ is suitable}}}{\max(\snorm{M_n^\ast\wt{b}}_2, \eps n)}\bigg] \pm O(\eps n^{-1})\\
&= \mb{E}\bigg[\frac{1}{\max(\snorm{M_n^\ast\wt{b}}_2, \eps n)}\bigg] \pm O(\eps n^{-1})
\end{align*}
proved that $n$ is sufficiently large with respect to $\eps$.

The next idea is to approximate $\snorm{M_n^\ast\wt{b}}_2$ via sampling random coordinates $j_1,\ldots,j_T$ for a sufficiently large value of $T$ and then estimating the $L^2$-norm of the vector $M_n^\ast\wt{b}$ via examining only these coordinates. This converts understanding a quadratic form into a question of purely linear forms. Let $\mc{E}_1$ denote the event that $\snorm{M_n^\ast\wt{b}}_3/n^{1/3}\le\eps^{-1}n^{1/2}$. \cref{lem:upper-tail} and Markov's inequality implies that $\mc{E}_1$ holds with probability at least $1-\eps^3$, hence 
\begin{equation}\label{eq:third-moment-truncate}
\mb{E}\bigg[\frac{1}{\max(\snorm{M_n^\ast\wt{b}}_2, \eps n)}\bigg]= \mb{E}\bigg[\frac{\mbm{1}_{\mc{E}_1}}{\max(\snorm{M_n^\ast\wt{b}}_2, \eps n)}\bigg] \pm O(\eps^2n^{-1}).
\end{equation}

Let $\psi\colon\mb{R}\to\mb{R}_{\ge 0}$ be smooth with $\psi(y)=0$ for $y\le\eps/2$ and $\psi(y)=y$ for $y\ge\eps$, $0\le\psi(y)\le y$ for $y\in[\eps/2,\eps]$, and such that $\psi$ is $O(1/\eps)$-lipschitz (a construction can be derived in a standard manner using bump functions). Let $g(x):=\psi(1/\max(x,\eps))$. Then
\begin{equation}\label{eq:tie-probability-1}
(8\pi\mr{Var}[\Delta])^{1/2}n\mb{P}[A\text{ ties }B]=\mb{E}[\mbm{1}_{\mc{E}_1}g(\snorm{M_n^\ast\wt{b}}_2/n)]\pm O(\eps^2)
\end{equation}
since $\mc{E}_1$ implies $\snorm{M_n^\ast\wt{b}}_2/n^{1/2}\le\eps^{-1}n^{1/2}$.

Now, by \cref{lem:resample-estimate}, given the event $\mc{E}_1$ a uniformly random sample of $T=\lfloor\eps^{-100}\rfloor$ independent coordinates $j_1,\ldots, j_T$ satisfies
\[\bigg|\snorm{M_n^\ast\wt{b}}_2^2 - \frac{n}{T}\sum_{1\le k\le T}\sang{\vec{e}_{j_k},M_n^\ast\wt{b}}^2\bigg|\le\eps^{20}n\]
with probability at least $1-\eps^{10}$. It follows immediately that
\begin{equation}\label{eq:transfer-linear}
\bigg|\mb{E}[\mbm{1}_{\mc{E}_1}g(\snorm{M_n^\ast\wt{b}}_2/n)]-\mb{E}\bigg[\mb{E}\bigg[g\bigg(\bigg(\frac{1}{nT}\sum_{k=1}^T\sang{M_n^\ast\wt{b},\vec{e}_{j_k}}^2\bigg)^{1/2}\bigg)\bigg|j_1,\ldots,j_T\bigg]\bigg]\bigg|\lesssim\eps^{10}+(1-\eps^{10})\eps^{19}+\eps^2,
\end{equation}
using that $g(y)\in[0,\eps^{-1}]$. Let $L=\lfloor\eps^{-1000}\rfloor$ and $n'=\lfloor\eps^{1000}n\rfloor$. For each $k$ let $\wt{j}_k$ denote the nearest index to $j_k$ in the set $\{n',2n',\ldots,Ln'\}$. Since $g$ is appropriately Lipschitz and applying \cref{lem:round}, we have
\begin{align*}
\bigg|\mb{E}\bigg[\mb{E}\bigg[g\bigg(\bigg(\frac{1}{nT}&\sum_{k=1}^T\sang{M_n^\ast\wt{b},\vec{e}_{j_k}}^2\bigg)^{1/2}\bigg)\bigg|j_1,\ldots,j_T\bigg]\bigg]-\mb{E}\bigg[\mb{E}\bigg[g\bigg(\bigg(\frac{1}{nT}\sum_{k=1}^T\sang{M_n^\ast\wt{b},\vec{e}_{\wt{j}_k}}^2\bigg)^{1/2}\bigg)\bigg|j_1,\ldots,j_T\bigg]\bigg]\bigg|\\
&\lesssim\eps^{100}.
\end{align*}

Now by L\'evy continuity, \cref{lem:crucial,lem:fourier-coeff}, and \cref{M1} we see that the distributions of
\[\bigg(\frac{\sang{M_n^\ast\wt{b},\vec{e}_{kn'}}}{\sqrt{n}}\bigg)_{1\le k\le L},\qquad\bigg(\frac{\sang{M_n^\ast\wt{X},\vec{e}_{kn'}}}{\sqrt{n}}\bigg)_{1\le k\le L}\]
converge jointly to a fixed distribution independent of $n$ (but depending on $\eps$). As $g$ is a bounded and continuous function, for $n$ sufficiently large by the Portmanteau theorem we deduce that 
\begin{align}
\bigg|\mb{E}\bigg[\mb{E}\bigg[g\bigg(\bigg(\frac{1}{nT}&\sum_{k=1}^T\sang{M_n^\ast\wt{b},\vec{e}_{\wt{j_k}}}^2\bigg)^{1/2}\bigg)\bigg|j_1,\ldots,j_T\bigg]\bigg]-\mb{E}\bigg[\mb{E}\bigg[g\bigg(\bigg(\frac{1}{nT}\sum_{k=1}^T\sang{M_n^\ast\wt{X},\vec{e}_{\wt{j_k}}}^2\bigg)^{1/2}\bigg)\bigg|j_1,\ldots,j_T\bigg]\bigg]\bigg|\notag\\
&\le \eps^2,\label{eq:transfer-gaussian}
\end{align}
say. Finally, using \cref{lem:resample-estimate,lem:round,lem:upper-tail} in the Gaussian model instead and mimicking the above argument (for \cref{eq:transfer-linear}) in reverse demonstrates
\begin{equation}\label{eq:transfer-linear-gaussian}
\bigg|\mb{E}[\mbm{1}_{\mc{E}_2}g(\snorm{M_n^\ast\wt{X}}_2/n)]-\mb{E}\bigg[\mb{E}\bigg[g\bigg(\bigg(\frac{1}{nT}\sum_{k=1}^T\sang{M_n^\ast\wt{X},\vec{e}_{\wt{j}_k}}^2\bigg)^{1/2}\bigg)\bigg|j_1,\ldots,j_T\bigg]\bigg]\bigg|\lesssim\eps^2
\end{equation}
where $\mc{E}_2$ is the event that $\snorm{M_n^\ast\wt{X}}_3/n^{1/3}\le\eps^{-1}n^{1/2}$. Combining \cref{eq:tie-probability-1,eq:transfer-linear,eq:transfer-gaussian,eq:transfer-linear-gaussian}, we deduce
\[(8\pi\mr{Var}[\Delta])^{1/2}n\mb{P}[A\text{ ties }B]=\mb{E}[\mbm{1}_{\mc{E}_2}g(\snorm{M_n^\ast\wt{X}}_2/n)]\pm O(\eps^2).\]
Using \cref{lem:upper-tail} for the Gaussian model and Markov's inequality, we easily find
\begin{equation}\label{eq:g-expression}
(8\pi\mr{Var}[\Delta])^{1/2}n\mb{P}[A\text{ ties }B]=\mb{E}[g(\snorm{M_n^\ast\wt{X}}_2/n)]\pm O(\eps^2).
\end{equation}
Finally, letting $W_\ell,W_\ell'\sim\mc{N}(0,\mr{Var}[\Delta])$ for $\ell\ge 1$ we see
\[\frac{\snorm{M_n^\ast\wt{X}}_2}{n}\overset{d.}{=}\bigg(\sum_{\ell=1}^{\lfloor n/2\rfloor}\frac{\sigma_{n,\ell}^2}{n^2}(W_\ell^2+W_\ell'^2)\bigg)^{1/2}\]
by a variant of the spectral theorem applied to $M_n^\ast$ and \cref{def:dice-matrix}. By \cref{M6,M9} and \cref{thm:concentration-hypercontractivity} we deduce
\begin{align*}
\mb{E}[g(\snorm{M_n^\ast\wt{X}}_2/n)]&=\mb{E}g\bigg(\bigg(\sum_{\ell=1}^{\lfloor n/2\rfloor}\frac{\sigma_{n,\ell}^2}{n^2}(W_\ell^2+W_\ell'^2)\bigg)^{1/2}\bigg)\\
&=\mb{E}g\bigg(\bigg(\sum_{\ell=1}^{\lfloor\eps^{-8}\rfloor}\frac{\sigma_{n,\ell}^2}{n^2}(W_\ell^2+W_\ell'^2)\bigg)^{1/2}\bigg)\pm O(\eps^2)\\
&=\mb{E}g\bigg(\bigg(\sum_{\ell=1}^{\lfloor\eps^{-8}\rfloor}\sigma_\ell^2(W_\ell^2+W_\ell'^2)\bigg)^{1/2}\bigg)\pm O(\eps^2)
\end{align*}
as long as $n$ is large in terms of $\eps$. We claim that taking the limit $\eps\to 0^+$ gives the result. To check this, we note that for all $\rho\ge 0$ we have
\[\mb{P}\big[\sum_{\ell\ge1}\sigma_\ell^2(W_\ell^2+W_\ell'^2)\le\rho\big]\le \mb{P}\bigg[\bigcap_{1\le\ell\le 5}\sigma_\ell^2(W_\ell^2+W_\ell'^2)\le \rho\bigg]\le\prod_{\ell=1}^5(\mb{P}[W_\ell^2\le K\rho]\mb{P}[W_\ell'^2\le K\rho])\lesssim\rho^5\]
for an appropriate absolute constant $K$. Additionally,
\[\mb{P}\big[\sum_{\ell\ge1}\sigma_\ell^2(W_\ell^2+W_\ell'^2)\ge\eps^{-1}\big]\le\exp(-\eps^{-\Omega(1)})\]
by \cref{thm:concentration-hypercontractivity}. Therefore, we can absorb the difference between $g(y)$ and $1/y$ without any issue, uniformly in the limit. That is,
\[\lim_{\eps\to 0^+}\mb{E}g\bigg(\bigg(\sum_{\ell=1}^{\lfloor\eps^{-8}\rfloor}\sigma_\ell^2(W_\ell^2+W_\ell'^2)\bigg)^{1/2}\bigg)=\mb{E}\bigg[\bigg(\sum_{\ell\ge 1}\sigma_\ell^2(W_\ell^2+W_\ell'^2)\bigg)^{-1/2}\bigg]\]
(note that $g$ depends on $\eps$ here). Combining these final equalities with \cref{eq:g-expression} and taking $\eps$ to go slowly to $0$, and letting $(W_\ell,W_\ell')=\sqrt{\mr{Var}[\Delta]}(Z_\ell,Z_\ell')$ for standard Gaussians $Z_\ell,Z_\ell'$, we ultimately deduce
\[(8\pi\mr{Var}[\Delta])^{1/2}n\mb{P}[A\text{ ties }B]=(\mr{Var}[\Delta])^{-1/2}\mb{E}\bigg[\bigg(\sum_{\ell\ge 1}\sigma_\ell^2(Z_\ell^2+Z_\ell'^2)\bigg)^{-1/2}\bigg]+o(1).\]
Rearranging, this agrees with the desired \cref{thm:ties}.
\end{proof}

We end by briefly discussing an (amusing) interpretation of the constant $\alpha$ corresponding to ordinary least squares regression in the context of Brownian motion. The above proof implicitly shows that given a fixed set of indices $j_1,\ldots,j_k$, $(\vec{e}_{j_t}M_n^\ast X)_{1\le t\le k}$ in distribution limits toward a snapshot of a Brownian motion at the times $j_t/n$ where the Brownian motion is conditioned to end at $0$ at time $1$ and conditioned to have total signed area under the Brownian motion equal to $0$. Note that then $\min_{a,b\in\mb{R}}(c_j-a-bj)^2$ corresponds to approximating such a Brownian motion by the best linear-function fit coming from Ordinary Least Squares regression. We leave making this precise an exercise for the reader.

\bibliographystyle{amsplain0.bst}
\bibliography{main.bib}

\end{document}